\documentclass{amsart}

\usepackage{amsmath,amsfonts,amssymb,amsthm}
\usepackage{hyperref}
\usepackage{mathtools}
\usepackage{pdfpages}
 
\usepackage[capitalize]{cleveref}
\usepackage{stmaryrd}
\usepackage{tikz-cd}
\usepackage{graphicx}
\usepackage{adjustbox}
\usepackage{seqsplit}
\usepackage{thmtools}

\newtheorem{lem}{Lemma}[section]
\newtheorem{teo}[lem]{Theorem}

\newtheorem{pro}[lem]{Proposition}
\newtheorem{cor}[lem]{Corollary}
\newtheorem{claim}[lem]{Claim}

\newtheorem*{con*}{Conjecture}

\newtheorem{Conj}{Conjecture}

\theoremstyle{definition}

\theoremstyle{remark}
\newtheorem*{rem*}{Remark}
\newtheorem{rem}[lem]{Remark}

\newcommand{\argu}{\hbox to 7truept{\hrulefill}}

\DeclareMathOperator{\D}{\mathcal D}

\DeclareMathOperator{\Id}{Id}

\DeclareMathOperator{\St}{St}

\DeclareMathOperator{\im}{Im}

\DeclareMathOperator{\Hom}{Hom}

\DeclareMathOperator{\Aut}{Aut}

\DeclareMathOperator{\Out}{Out}

\DeclareMathOperator{\ab}{ab}

\DeclareMathOperator{\Tor}{Tor}

\DeclareMathOperator{\Ext}{Ext}
\DeclareMathOperator{\lcm}{lcm}
\DeclareMathOperator{\length}{length}

\DeclareMathOperator{\prd}{pd}

\newcommand{\Z}{\mathbb{Z}}

\newcommand{\N}{\mathbb{N}}
\newcommand{\CC}{\mathbb{C}}
\newcommand{\Q}{\mathbb{Q}}

\newcommand{\cd}{\text{cd}}

\DeclareMathOperator{\FP}{\mathtt{FP}}

\newcommand{\normal}[1]{\langle\!\langle #1 \rangle\!\rangle}

\newcounter{marcocomments}

\newcounter{andreicomments}

\newcounter{pablocomments}

 \date{\today}
 
\begin{document}
\title  {Group pairs, coherence and Farrell-Jones Conjecture for $K_0$}
\author{Andrei Jaikin-Zapirain}
 \address{Instituto de Ciencias Matem\'aticas, CSIC-UAM-UC3M-UCM}
 \email{andrei.jaikin@icmat.es}
 \author{Marco Linton}
 \address{Instituto de Ciencias Matem\'aticas, CSIC-UAM-UC3M-UCM}
 \email{marco.linton@icmat.es}
 \author{Pablo S\'anchez-Peralta}
  \address{Departamento de Matem\'aticas, Universidad Aut\'onoma de Madrid}
    \email{pablo.sanchezperalta@uam.es}

\begin{abstract}

A  {group pair} $(G, X)$ consists of a group $G$ together with a $G$-set $X$.  
Such a pair encodes properties of $G$ relative to the stabilisers of points in $X$.  
In this paper, we show how to combine properties of group pairs and their stabilisers to prove coherence results for $G$ and its group algebra, as well as to study the quotient of $G$ obtained by killing the stabilisers.  

In particular, we prove that a torsion-free one-relator product of locally indicable groups is coherent provided that both factor groups are coherent. Moreover, we show that the group algebra of such a group over a field of characteristic $0$ is coherent whenever the group algebras of the factors are coherent.  

As other consequences of our methods, we also show that extensions of coherent locally indicable hyperbolic groups by $\Z$ are coherent and that groups admitting a Cohen–Lyndon presentation satisfy the Farrell–Jones Conjecture for $K_{0}$.

\end{abstract}

\maketitle

\section{Introduction} 
\subsection{Homological coherence and coherence of groups}
A group is called {\bf coherent} if all its finitely generated subgroups are finitely presented. This property has attracted significant attention in recent years, as reflected in a survey by Wise \cite{Wi20}.   In this paper, we prove the coherence of one-relator products of coherent locally indicable groups. 

\begin{teo} \label{main}
Let $A$ and $B$ be two locally indicable coherent groups and let $w\in A*B$ be an element that is not conjugated to an element in $A$ or $B$. Then the group $A*B/\normal{w}$ is coherent.
\end{teo}

This theorem generalizes the case of one-relator groups, first proved for one-relator groups with torsion in \cite{LW20,Wi22a}, and in the general case in \cite{JL23}. When the word $w$ is a proper power, \cref{main} is due to Howie--Short \cite{HH23}. In this paper, we prove Theorem \ref{main} in the case where $w$ is not a proper power.  

Our proof of \cref{main} also applies to the case in which $w$ is a proper power, provided the one-relator product satisfies the weak Atiyah conjecture. We discuss the weak Atiyah conjecture in \cref{sec: wAc}. Let us only mention that linear groups over a field of characteristic~$0$ or virtually locally indicable groups satisfy the weak Atiyah conjecture. 

As in \cite{JL23}, we divide the proof of Theorem \ref{main} into two steps. Recall that a group is called {\bf homologically coherent over a ring $R$} if all its finitely generated subgroups are of type $\FP_2(R)$.  Our first step consists in proving that $A * B / \langle\!\langle w \rangle\!\rangle$ is homologically coherent over~$\Q$. For one-relator groups, this was shown in \cite{JL23} using the fact that the second $L^2$-Betti number of a one-relator group $G$ is trivial, and that $\mathrm{cd}_{\Q}(G) \leq 2$. In our setting, these two properties do not hold; instead, we replace them with analogous properties for group pairs. We refer the reader to Sections~\ref{sect:gp} and \ref{sect:L2}  for all relevant definitions.

\begin{teo}\label{cohomologicalcoherence}
Let $G$ be a group satisfying the weak Atiyah conjecture. Let $(G,X)$ be a group pair, and assume that:
\begin{enumerate}
    \item For every $x \in X$, the stabiliser $G_x$ is homologically coherent over~$\Q$,
    \item $\cd_{\Q}(G, X) \leq 2$ and 
    \item $b_2^{(2)}(G, X) = 0$.
\end{enumerate}
Then $G$ is homologically coherent over~$\Q$.
\end{teo}

In the case of a torsion-free one-relator product $G = A * B / \normal{w}$ of locally indicable groups $A$ and $B$, we apply \cref{cohomologicalcoherence} for $X = G/A \sqcup G/B$.

The second step in the proof of Theorem~\ref{main} consists of promoting homological coherence over~$\mathbb{Q}$ for $A * B / \langle\!\langle w \rangle\!\rangle$ to  coherence. For one-relator groups, this was achieved using the Magnus hierarchy. A variation of this method, incorporating results from \cite{Li24}, can also be applied in our context. However, we present an alternative approach based on the Cohen--Lyndon property of group pairs (see Section~\ref{sect:prom} for definitions and details). Since the group pair $(A * B, A * B / \langle w \rangle)$ satisfies the Cohen--Lyndon property by work of Edjvet--Howie \cite{EH87}, we conclude that the homological coherence over $\Q$ of the group $G$ implies its coherence. 

\begin{teo}\label{promotion}
Let $G$ be a coherent group, and let $\mathcal{P} = (G, X)$ be a group pair satisfying the Cohen--Lyndon property, such that for every $x \in X$, the stabiliser $ G_x$ is infinite cyclic. Then every subgroup of the quotient group
\[
G \big/ \left\langle G_x : x \in X \right\rangle
\]
that is of type $\mathrm{FP}_2(\mathbb{Q})$ is finitely presented.
\end{teo}

During the proof of \cref{main}, we also show that intersections of finitely generated subgroups $H\leqslant A*B/\normal{w}$ with the factors $A$ and $B$ are themselves finitely generated. This is stated as \cref{teo:intersection}. We conjecture that the same property holds when $w$ is allowed to be a proper power, our only obstacle is that we do not know whether such groups satisfy the weak Atiyah conjecture. It would also be interesting to determine which conditions ensure that the number of double cosets $AgH$ with $H \cap A^g$ nontrivial is finite. Our techniques allow us to prove that if $H$ is of type $\mathrm{FP}_2(\mathbb{Q})$, then for all but finitely many double cosets $AgH$, the intersection $H \cap A^g$ is cyclic.

Wise conjectured \cite[Conjecture 7.4]{Wi20} that any extension of a coherent hyperbolic group with $\Z$ is coherent. As another application of \cref{cohomologicalcoherence}, we partially resolve his conjecture.

\begin{teo}
\label{hyperbolic_extension}
Let $G \cong H\rtimes\Z$ be a group satisfying the weak Atiyah conjecture. If $H$ is hyperbolic and (homologically) coherent (over $\Q$), then $G$ is (homologically) coherent (over $\Q$).
\end{teo}

We remark that the weak Atiyah conjecture (over $\mathbb{C}$) is closed under infinite cyclic extensions by \cite[Lemma 4.1]{SP_Atiyah}. In particular, \cref{hyperbolic_extension} implies that if $H$ is a hyperbolic and virtually locally indicable coherent group, then $H\rtimes \Z$ is coherent.

\subsection{Coherence of group algebras}
Recall that a ring is called \textbf{coherent} if all its finitely generated left ideals are finitely presented. In \cite{JL23}, it was proven that if $K$ is a field of characteristic 0, the group algebra $KG$ is coherent for one-relator groups $G$. Here, we extend this result to certain one-relator products.

\begin{teo}\label{teo:coherencegroupalgebras}
    Let $K$ be a field of characteristic 0, $A$ and $B$   two locally indicable groups, and   $w \in A * B$   an element that is neither conjugated to an element in $A$ or $B$, nor a proper power. Let $G = A * B / \normal{w}$. Assume that $KA$ and $KB$ are coherent. Then $KG$ is coherent.
\end{teo}
It seems reasonable to conjecture that the same conclusion holds when $w$ is a proper power in $A * B$. 
This occurs, for example, in the case of one-relator groups with torsion. 
The only obstacle for our proof of \cref{teo:coherencegroupalgebras} to work for general one-relator products of locally indicable groups 
is that the weak Atiyah conjecture is not known for them.

\subsection{\texorpdfstring{$K_0$}{K0} of group algebras}
Let $S$ be a ring with unit, {\bf the projective class group} $K_0(S)$ is the free abelian group on finitely generated projective left $S$-modules modulo the relation $[P] = [P_1] + [P_2]$ if there is a short exact sequence of left $S$-modules $0\rightarrow P_1 \rightarrow P \rightarrow P_2 \rightarrow 0$. Here $[P]$ denotes the element of $K_0(S)$ corresponding to the projective left $R$-module $P$.

A commutative   ring $R$ is called {\bf   regular} if it is  Noetherian and all its left $R$-modules are of type $\mathtt{FP}$. A central open question in the $K_0$ theory of group rings is the Farrell-Jones Conjecture for $K_0(RG )$ for a torsion-free group $G$ and a regular ring $R$.

\begin{Conj}\label{conj: farrell-jones}
    Let $G$ be a torsion-free group and let $R$ be a   regular ring. Then the map
    \[
        K_0(R) \longrightarrow K_0(RG )
    \]
    induced by the inclusion $R \rightarrow RG $ is an isomorphism.
\end{Conj}

One of the first classes of groups that were known to satisfy \cref{conj: farrell-jones} was the class of free groups, a result due to Gersten \cite{GerstenK0freeprod}. Nowadays our knowledge of \cref{conj: farrell-jones} is much larger, we know, for instance, that it holds for hyperbolic groups \cite{BLR_FJhyp}. We refer the reader to \cite{LuckKLtheory} for the current status of the conjecture.

The Farrell-Jones Conjecture for $K_0(RG )$ is related to  several other  conjectures in the theory of group rings: the Kaplansky Idempotent Conjecture \cite[Theorem 1.12]{BLR_FJapplications}, the Weak  and Strong Bass conjectures \cite [Conjectures 4.4 and 4.5]{BassConjectures} and  the Base Ring Conjecture \cite[Conjecture 85]{LR_BaumConnes}.

We want to highlight an important topological implication. A connected topological space $X$ is {\bf finitely dominated} if there exists a finite CW-complex $Y$ and maps $i:X\rightarrow Y$ and $r: Y \rightarrow X$ so that $r\circ i$ is homotopic to the identity. The following result is known as Wall's obstruction.

\begin{teo}[{\cite[Theorem F]{WallObstruction}}]
    Let $X$ be finitely dominated, and suppose that \cref{conj: farrell-jones} holds for $\Z[\pi_1(X,x)]$. Then $X$ is homotopy equivalent to a finite CW-complex.
\end{teo}

In this paper, we prove the following result, which implies \cref{conj: farrell-jones} for torsion-free groups with a presentation satisfying the Cohen--Lyndon property. Interestingly, Arenas and Duda showed in \cite{ArCL} that some non-metric small cancellation groups like $C(6)$, $C(4) - T (4)$, and $C(3) - T (6)$ admit a Cohen-Lyndon presentation, thus providing non-hyperbolic examples. Our result describes the projective class group for these groups.

\begin{teo}     
    \label{teoK_0_intro}
Let $\mathcal{P} = (G, X)$ be a group pair satisfying the Cohen--Lyndon property such that, for every $x \in X$, the stabiliser $G_x$ coincides with its normalizer in $G$.    
    Let $R$ be a regular ring, and assume that $\cd_R(G) < \infty$ and that the group ring $RG$ is coherent.  
    Then the natural map
    \[
        K_0(RG)  
        \longrightarrow K_0\left(R\left [
G \big/ \left\langle G_x : x \in X \right\rangle
\right ]\right) \] 
    is surjective.
\end{teo}

The paper is organised as follows. In \cref{sec: prelim} we explain the preliminary results used in the paper. In particular, we introduce the weak Atiyah conjecture. In \cref{sect:gp} we define the notion of group pairs and study the relationship between finiteness properties of the group and its stabilisers. In \cref{sect:L2} we introduce $L^2$-Betti numbers of group pairs and show \cref{cohomologicalcoherence}. We also prove in \cref{teo:intersection} that intersections of finitely generated subgroups $H\leqslant A*B/\normal{w}$ with the locally indicable factors $A$ and $B$ are themselves finitely generated. \cref{sect:prom} is devoted to the proof of \cref{promotion}. As an application we prove \cref{main}. We finish this section with a module theoretic reinterpretation of the Cohen--Lyndon property. \cref{sec:extensions} introduces the notion of a graph of group pairs and develops a tool to prove group coherence of infinite cyclic extensions by looking at its sub-extensions of maximal one-ended subgroups. In particular, we establish \cref{hyperbolic_extension}. In \cref{sec: coh_group_alg} we study modules over group pairs and show in \cref{coherencegroupalgebraspais} coherence of group algebras associated to certain group pairs whose stabilisers have coherent group algebras; this shows \cref{teo:coherencegroupalgebras}. We also pose some conjectures on the coherence of group algebras of certain infinite cyclic extensions. We finish the paper with \cref{sec: FJ_conj} where we prove \cref{teoK_0_intro}.

\subsection*{Acknowledgments}

We would like to thank Peter Kropholler  for useful discussion  on modules of finite projective dimension. 

The authors would like to thank the Isaac Newton Institute for Mathematical Sciences, Cambridge, for support and hospitality during the programme Operators, Graphs, Groups, where part of work on this paper was undertaken. This work was partially supported by EPSRC grant EP/Z000580/1 and the grant \seqsplit{PID2020-114032GB-I00/AEI/10.13039/501100011033} of the Ministry of Science and Innovation of Spain. The first author  was partially supported by a grant from the Simons Foundation. The second author was supported by the grant 202450E223 (Impulso de líneas científicas estratégicas de ICMAT).

\section{Preliminaries} \label{sec: prelim}

\subsection{General notation} 
All rings considered in this paper are assumed to be unitary, and all modules are left modules unless stated otherwise. We reserve the letter $R$ for commutative rings and $K$ for fields.

If $S$ is a ring, the {\bf length} of an $S$-module $M$ is the supremum of the lengths of chains of submodules, denoted by $\length_S(M)$. In the case of a module over a division ring, the length coincides with its dimension. Over an Artinian ring, every finitely generated module has finite length.

Let  $G$ be a group and let $M_1$ and $M_2$ be two $RG$-modules. Then $M_1\otimes_R  M_2$ is an $RG$-module, where the action of elements of $G$ is defined as $$g\cdot (m_1\otimes m_2)=(g\cdot m_1)\otimes (g\cdot m_2), \quad \text{for all } m_1 \in M_1, m_2 \in M_2, \ g \in G.$$
If  $H$ is a subgroup of $G$ and $L$ is an $RH$-module, we denote by ${}^GL$ the induced $RG$-module $RG\otimes_{RH} L$.

If $M$ is a left $RG$-module, we define $M^{op}$ to be the {\bf opposite} right $RG$-module: $M^{op} = M$ as a set, and the action is given by
\[
m \cdot g := g^{-1} m, \quad \text{for all } m \in M, \ g \in G.
\]

If a group $G$ acts on a set $X$ and $x \in X$, then we denote by $G_x$ 
the stabiliser of $x$ in $G$.

If $G$ is a group and $g, h\in G$, our convention for conjugation of $g$ by $h$ will be to write $g^h := h^{-1}gh$.

\subsection{Finiteness Conditions on Modules}
Let $S$ be a ring and $M$ an $S$-module. We say that $M$ is of \textbf{type $\FP_k$} ($k \geq 0$) if there exists an exact sequence of finitely generated projective $S$-modules
\begin{equation}\label{projres}
P_k \to \cdots \to P_1 \to P_0 \to M \to 0.
\end{equation}
We say that $M$ is of \textbf{type $\FP_\infty$} if it is of type $\FP_k$ for all $k \geq 0$, and that $M$ is of \textbf{type $\FP$} if it is of type $\FP_{\infty}$ and the projective dimension $\operatorname{pd}_S(M)$ is finite; that is, $M$ admits a resolution of the form \eqref{projres} for $k = \prd_S(M)$ with trivial first kernel. Note that if $S$ is coherent, then any   $S$-module of type $\FP_1$ is of type $\FP_\infty$.  

To study the group $K_0(S)$, we introduce another group, denoted by $G_0(S)$. Given a short exact sequence of $S$-modules
\[
0 \to M_1 \to M_2 \to M_3 \to 0,
\]
it follows from \cite[Proposition 1.4 and Proposition 4.1b]{Bieri_Notes} that if two of the modules $M_1$, $M_2$ or $M_3$ are of type $\FP$, then so is the third one.

The group $G_0(S)$ is the free abelian group generated by symbols $[P]$, where $P$ is an $S$-module of type $\FP$, subject to the relations
\[
[P_2] = [P_1] + [P_3]
\]
whenever there is a short exact sequence
\[
0 \to P_1 \to P_2 \to P_3 \to 0.
\]

The following lemma is well known (see, for example, \cite[Theorem 3.1.13]{Rosenberg_Kbook}):

\begin{lem}\label{K0G0}
Let $S$ be a ring. Then the natural map
\[
\kappa_S: K_0(S) \to G_0(S), \quad [P] \mapsto [P],
\]
is an isomorphism of groups.
\end{lem}

Finally, we will use the following lemma (see, for example, \cite{PK25}):

\begin{lem}\label{critfinitepd} 
Let $R$ be a regular ring, and let $G$ be a group with $\cd_R(G)<\infty$. Then for any $RG$-module $M$ it holds that $$\prd_{RG}(M)\le \cd_{R}(G)+\prd_R(M).$$
\end{lem}

\subsection{One-relator products}

Let $A$ and $B$ be groups and let $w\in A*B$ be an element. The length of $w$ will be understood to be the length of $w$ as a word over $A\cup B$. We say $w$ is reduced or cyclically reduced if it is reduced or cyclically reduced as a word over $A\cup B$ respectively. A prefix of $w$ is a word $u$ so that $w = uv$ for some $v$ and the length of $w$ is the length of $u$ plus the length of $v$. We say $u$ is a proper prefix if $v$ has positive length. A (proper) suffix and a (proper) subword are defined similarly. A non-empty word $w$ is not a proper power if there is no word $u$ and integer $n\geqslant 2$ so that $u^n = w$ in $A*B$.

If $w$ is cyclically reduced and of length at least two, the quotient group
\[
G = \frac{A*B}{\normal{w}}
\]
is called the {\bf one-relator product}. Without strong conditions on $A$, $B$ or $w$, it is very difficult to say anything about this group. One condition that can be put on $A$ and $B$ that yields a lot of structure on $G$ is local indicability.

We collect below some known statements about one-relator products of locally indicable groups. The statements are all due to Howie, see \cite{howie_81} for the first and \cite{Ho82} for the other two.

\begin{teo}
\label{one-relator_facts}
Let $A$ and $B$ be locally indicable groups, let $w\in A*B$ be a cyclically reduced element of length at least two and let $G$ be the one-relator product. Then:
\begin{enumerate}
\item $A$ and $B$ embed into $G$.
\item If $u, v$ are distinct proper prefixes of $w$, then $u\neq_G v$.
\item If $w$ is not a proper power in $A*B$, then $G$ is locally indicable.
\end{enumerate}
\end{teo}

We shall also need the following theorem of Howie's \cite[Theorem 11]{Ho84} which is a generalisation of Lyndon's identity theorem.

\begin{teo}
\label{identity_theorem}
Let $A$ and $B$ be locally indicable groups, let $u\in A*B$ be a cyclically reduced word that is not a proper power, $w = u^n$ and let $G = \frac{A*B}{\normal{w}}$. If $N = \normal{w}$, then $N_{\ab}\cong \Z [G/\langle u\rangle]$ as a $\Z G$-module.
\end{teo}

\subsection{\texorpdfstring{$L^2$}--Betti numbers of modules} \label{sec: wAc}
Let $G$ be a countable group and let $\ell^2(G)$ denote the Hilbert space with Hilbert basis the elements of $G$, that is, $\ell^2(G)$ consists of all square-summable formal sums
\[
\sum_{g \in G} a_g g
\]
with $a_g \in \mathbb{C}$, and inner product
\[
\left\langle \sum_{g \in G} a_g g, \sum_{g \in G} b_g g \right\rangle = \sum_{g \in G} a_g \overline{b_g}.
\]
The left and right multiplication actions of $G$ on itself extend to left and right actions of $G$ on $\ell^2(G)$. The right action of $G$ on $\ell^2(G)$ further extends to an action of $\mathbb{C}G$ on $\ell^2(G)$, and hence we obtain that the group algebra $\mathbb{C}G$ acts faithfully as bounded linear operators on $\ell^2(G)$.

The von Neumann algebra $\mathcal{N}(G)$ is the ring of bounded operators on $\ell^2(G)$ which commute with the left action of $G$. We consider $\mathbb{C}G$ as a subalgebra of $\mathcal{N}(G)$. The ring $\mathcal{N}(G)$ satisfies the left and right Ore conditions (a result proved by S.~K.~Berberian in \cite{Be82}), and its classical ring of fractions is denoted by $\mathcal{U}(G)$. The ring $\mathcal{U}(G)$ can also be described as the ring of densely defined closed (unbounded) operators which commute with the left action of $G$.

The computation of $L^2$-Betti numbers has been algebraized through the seminal works of L\"uck \cite{Lu88I, Lu88II} and the thesis of Reich \cite{Re98}. The basic observation is that one can use a dimension function $\dim_{\mathcal{U}(G)}$, which is defined for all modules over $\mathcal{U}(G)$, and compute the $k$-th $L^2$-Betti number of a $\mathbb{C}G$-module $M$ using the following formula:
\[
\beta^{\mathbb{C}G}_k(M) = \dim_{\mathcal{U}(G)} \operatorname{Tor}^{\mathbb{C}G}_k(\mathcal{U}(G), M).
\]
We recommend the books \cite{Lu02book, Kam19} and the survey \cite{Ja19survey} for the definition of $\dim_{\mathcal{U}(G)}$ and its properties.

The ring $\mathcal{U}(G)$ is an example of a $*$-regular ring. Already in the case $G = \langle t \rangle \cong \mathbb{Z}$ it is quite complicated as a ring (it is isomorphic to the ring of measurable functions on $S^1$). Therefore, it is sometimes more convenient to consider a smaller object $\mathcal{R}_{\mathbb{C}G}$, introduced by Linnell and Schick \cite{LS12}.

 We define $\mathcal{R}_{\CC G}$ as the $*$-regular closure of $\CC G$ in $\mathcal{U}(G)$, i.e., $\mathcal{R}_{\CC G}$ is the smallest $*$-regular subring of $\mathcal{U}(G)$ that contains $\CC G$. 
 We can also define a dimension function $\dim_{\mathcal{R}_{\CC G}}$ on $\mathcal{R}_{\CC G}$-modules and use it to define the $L^2$-Betti numbers (see \cite{Ja19survey}). The object $\mathcal{R}_{\CC G}$ is much simpler than $\mathcal{U}(G)$. For example, in the case $G = \langle t \rangle \cong \mathbb{Z}$, $\mathcal{R}_{\CC G}$ is isomorphic to $\CC(t)$, and $\dim_{\mathcal{R}_{\CC G}}$ is the usual dimension of $\CC(t)$-vector spaces.

 Let $K$ be a subfield of $\CC$ and  $M$   a $K G$-module, then its $L^2$-Betti numbers are computed using the formula
\[
\beta^{K G}_k(M) = \dim_{\mathcal{R}_{\CC G}} \operatorname{Tor}^{\CC G}_k(\mathcal{R}_{\CC G}, M).
\]

The \textbf{strong Atiyah conjecture} (over $K$) predicts that if $\operatorname{lcm}(G)$, the least common multiple of the orders of finite subgroups of $G$, is finite, then for every $KG$-module $M$
\[
\beta^{{K}G}_k(M) \in \frac{1}{\operatorname{lcm}(G)} \mathbb{Z}_{\ge 0} \cup \{\infty\}.
\]
We will use the fact that the strong Atiyah conjecture has been proved for locally indicable groups \cite{JL20}. In this case $\lcm(G)=1$ and $\mathcal{R}_{\CC G}$ is a division ring.
However, in this paper we will actually rely in most of the situations on a weaker version of the Atiyah conjecture.

We say that a group $G$ satisfies the \textbf{weak Atiyah conjecture} (over $K$)  if there exists $l \in \mathbb{N}$ such that for every $KG$-module $M$ and every $k$,
\[
\beta^{{K}G}_k(M) \in \frac{1}{l} \mathbb{Z}_{\ge 0} \cup \{\infty\}.
\]

When $K = \mathbb{C}$, this is equivalent to the ring   $\mathcal{R}_{\mathbb{C}G}$ being a semisimple (and so,  Artinian) algebra. There exist groups for which the weak Atiyah conjecture is known to hold, while the strong version remains open. This distinction arises because a group that virtually satisfies the weak Atiyah conjecture automatically satisfies it itself, whereas this inheritance property is not known for the strong Atiyah conjecture.  Thus, for example, we know that the weak Atiyah conjecture holds for finitely generated  groups that are linear over $\mathbb{C}$, but the strong Atiyah conjecture remains open for them (see, for example, Proposition 11.4, Theorem 12.7 and Question 12.8 from \cite{Ja19survey}).

Every finitely generated field $K$ of characteristic zero can be embedded into $\mathbb{C}$. Any such embedding induces a definition of the $L^2$-Betti numbers $\beta^{K G}_k(M)$. It was conjectured in \cite{Ja19} (and proved for sofic groups) that $\beta^{K G}_k(M)$ does not depend on the choice of embedding. For example, this is known to hold for locally indicable groups \cite{JL20}. This allows one to define $\beta^{K G}_k(M)$ for locally indicable groups and  any field $K$ of characteristic zero. In fact, the solution of the strong Atiyah conjecture for locally indicable groups imply that there exists a division ring $\mathcal{D}_{K G}$ such that
\[
\beta^{K G}_k(M) = \dim_{\mathcal{D}_{K G}} \operatorname{Tor}^{K G}_k(\mathcal{D}_{K G}, M).
\]
We recommend the reader to read the preliminaries of \cite{JL23} to find more information on the division ring $\mathcal{D}_{K G}$.
We will use the following result about $\mathcal{D}_{K G}$ proved in \cite{JL23}.

\begin{pro}\label{preflat}
    Let $G$ be a locally indicable group, $K$ a field of characteristic zero, and $M$ a right one-relator $KG$-module. Then the right $KG$-module $\mathcal{D}_{KG} \otimes_{K} M$ is flat.
\end{pro}

\begin{proof}
Since $M$ is one-relator, there exists a free right $KG$-module $L$ and $l \in L$ such that 
\[ M \cong L / l \cdot KG. \]
Without loss of generality, we may assume that $l \neq 0$. Then there exists a decomposition 
\[ L = L_0 \oplus KG, \]
where $L_0$ is a free right $KG$-submodule of $L$, and the projection $a$ of $l$ in the summand $KG$ is nontrivial. Hence, 
\[ l \cdot KG \cap L_0 = \{0\}. \]
Thus, we obtain an exact sequence of right $KG$-modules:
\[
0 \longrightarrow \mathcal{D}_{KG} \otimes_{K} L_0 \longrightarrow \mathcal{D}_{KG} \otimes_{K} M \longrightarrow \mathcal{D}_{KG} \otimes_{K} (KG / aKG) \longrightarrow 0.
\]

Since $L_0$ is free, the right $KG$-module $\mathcal{D}_{KG} \otimes_{K} L_0$ is free. On the other hand, by \cite[Lemma 5.1]{JL23}, for every left $KG$-module $V$, we have that
\[
{\rm Ann}_{V \otimes_K \mathcal{D}_{KG}}(a) = 0.
\]
However,
\[
\Tor_1^{KG}\!\left( \mathcal{D}_{KG} \otimes_{K} (KG / aKG), V\right) 
   \cong {\rm Ann}_{V \otimes_K \mathcal{D}_{KG}}(a).
\]
Hence, $\mathcal{D}_{KG} \otimes_{K} (KG / aKG)$ is flat. 

A flat-by-flat module is flat, and therefore $\mathcal{D}_{KG} \otimes_{K} M$ is flat.
\end{proof}

We will also need the following result. 

\begin{pro} \label{vanishing}
Let $G$ be a group and let $N \trianglelefteq G$ be a normal subgroup such that $G/N$ is infinite amenable. Let $M$ be a finitely generated $\mathbb{Q}G$-module. Assume that $\beta_1^{\mathbb{Q}N}(M)$ is finite. Then $\beta_1^{\mathbb{Q}G}(M) = 0$.
\end{pro}

\begin{proof}
We can represent $\mathbb{Q}G$ as a crossed product ring $\mathbb{Q}N * G/N$ (where $G$ acts by conjugation on $\Q N$), and define $S = \mathcal{U}(N) * G/N$. Since $M$ is finitely generated, there exists an exact sequence
\[
0 \to U \to (\mathbb{Q}G)^d \to M \to 0.
\]
This induces an exact sequence
\[
0 \to \Tor_1^{\mathbb{Q}G}(S, M) \to S \otimes_{\mathbb{Q}G} U \to S^d \to S \otimes_{\mathbb{Q}G} M \to 0.
\]

For any $S$-module $L$, define
\[
\dim L := \dim_{\mathcal{U}(G)}(\mathcal{U}(G) \otimes_S L).
\]
By \cite[Corollary 12.2 and Theorem 8.2]{Ja19survey}, this dimension function is exact. Thus, we have
\[
\dim \Tor_1^{\mathbb{Q}G}(S, M) = \dim_{\mathcal{U}(G)} \Tor_1^{\mathbb{Q}G}(\mathcal{U}(G), M) = \beta_1^{\mathbb{Q}G}(M).
\]

Now observe that
\[
\Tor_1^{\mathbb{Q}G}(S, M) \cong \Tor_1^{\mathbb{Q}N}(\mathcal{U}(N), M).
\]
Therefore,
$$  \dim_{\mathcal{U}(N)} \Tor_1^{\mathbb{Q}G}(S, M)=\beta_1^{\mathbb{Q}N}(M)$$ is finite.
By \cite[Theorem 5.1]{Pe_L2}, this implies that $\dim \Tor_1^{\mathbb{Q}G}(S, M) = 0$. Therefore, $\beta_1^{\mathbb{Q}G}(M) = 0$.
\end{proof}

\subsection{Graphs of groups and groups acting on trees}

We shall need some useful facts about graphs of groups and groups acting on trees. The reader is directed to Serre's book \cite{serre_80} and Bass' article \cite{Ba93} for the necessary background.

Recall that a graph of groups is a tuple 
\[
\mathcal{G} = (\Gamma, \{G_v\}_{v\in V(\Gamma)}, \{G_e\}_{e\in E(\Gamma)}, \{\partial_e^{\pm}\}_{e\in E(\Gamma)}\}
\]
where $\Gamma$ is a graph, the groups $G_v$ are the vertex groups, the groups $G_e$ are the edge groups and the maps $\partial_e^{\pm}\colon G_e\to G_{e^{\pm}}$ are monomorphisms. Here we use $e^{+}$ to denote the target vertex of $e$ and $e^-$ the origin vertex of $e$. We fix an orientation $E^+\subset E(\Gamma)$ and a spanning tree $T\subset \Gamma$.

The \textbf{fundamental group} $G=\pi_1(\mathcal G, T)$ of $\mathcal G$ with respect to $T$ is the group with presentation:
 \begin{itemize}
 	\item generators $\{G_v, t_e: v\in V(\Gamma), e\in E^+\}$;
 	\item relations of each $G_v, v\in V(\Gamma)$;
 	\item relations $(\partial_e^-(g))^{t_e}=\partial^+_e(g)$ for each $g\in G_e, e\in E^+$; 
 	\item relations $t_e=1$ if $e\in E(T)$ 	.
\end{itemize}
This group does not depend on the choice of $T$ or the orientation.

From a given graph of groups $\mathcal{G}$ one can construct a tree $\mathcal{T}$ on which $\pi_1(\mathcal{G}, T)$ acts without edge inversions, called the {\bf Bass--Serre tree}. Conversely, from a group action $G\curvearrowright \mathcal{T}$ on a tree, called a {\bf $G$-tree}, one can define the {\bf quotient graph of groups} $\mathcal{G}$ so that $G \cong \pi_1(\mathcal{G}, T)$ (where the underlying graph of $\mathcal{G}$ is $G\backslash \mathcal{T}$) and so that the Bass--Serre tree for $\mathcal{G}$ is $G$-equivariantly isomorphic to $\mathcal{T}$. Importantly, the vertex and edge stabilisers of $\mathcal{T}$ are conjugates of the vertex and edge groups of $\mathcal{G}$.

Important examples of groups acting on trees are the following:
\begin{enumerate}
    \item If $G$ is an infinitely ended group of type $\FP_2(\Z)$, then $G$ acts non-trivially and co-compactly on a tree $\mathcal{T}$ so that each edge stabiliser is finite and each vertex stabiliser has at most one end. The existence of such a $G$-tree is a theorem of Dunwoody \cite{dunwoody_85}.
    \item If $G$ is a one-ended hyperbolic group that is not co-compact Fuchsian, then $G$ acts co-compactly on a tree $\mathcal{T}$ so that each edge stabiliser is 2-ended (and so virtually $\Z$) and each vertex stabiliser either has finite outer automorphism group (relative to adjacent edge groups) or is virtually free. This is known as the JSJ-tree for $G$ and is canonical in the sense that any automorphism of $G$ induces a $G$-equivariant isomorphism of $\mathcal{T}$ (and so $G\rtimes_{\psi}\Z$ acts on $\mathcal{T}$ if $\psi\in \Out(G)$). The existence of such a tree is a theorem of Bowditch \cite{Bo98}.
\end{enumerate}
We shall use both of these decompositions in \cref{sec:extensions}.

The following theorem of Chiswell \cite[Theorem 1]{Ch76} allows one to compute (co)homology of (fundamental groups of) graphs of groups in terms of the (co)homology of the vertex and edge groups. In \cref{sec:graphs_of_group_pairs} we shall extend Chiswell's result to the setting of graphs of group pairs.

\begin{teo}
\label{Chiswell}
    Let $\mathcal{G}$ be a graph of groups as above and let $R$ be a ring. The following sequence is exact:
    \begin{equation*} 
    \begin{tikzcd}
    0 \arrow[r] & \bigoplus_{e\in E^+}RG\otimes_{RG_e}R \arrow[r, "\delta"]               & \bigoplus_{v\in V(\Gamma)}RG\otimes_{RG_v}R \arrow[r, "\epsilon"]    & R \arrow[r] 
             & 0 
    \end{tikzcd}
    \end{equation*}
    where     $\epsilon$ is the augmentation map and  $\delta$ is given by
    \[
    \delta(s\otimes 1_e) = s\cdot t_e\otimes 1_{e^+}-s\otimes 1_{e^-}  .
    \]
\end{teo}

Finally, the following proposition and its corollary will be useful in the proof of \cref{rel_hom_coherence}.

\begin{pro} \label{subtree}
Let $G$ be a group and let $\mathcal{T}$ be a $G$-tree with trivial edge stabilisers. Let $H\leqslant G$ be a subgroup and $\mathcal{S}\subset \mathcal{T}$ an $H$-invariant subtree such that the induced map of graphs $H\backslash \mathcal{S}\to G\backslash \mathcal{T}$ is injective. If $N\leqslant G$ is a subgroup acting freely on $\mathcal{T}$, then $N$ is free and $N\cap H$ is a free factor of $N$.
\end{pro}

\begin{proof}
The induced map of graphs $(N\cap H)\backslash \mathcal{S} \to N\backslash \mathcal{T}$ is injective since $H\backslash \mathcal{S}\to G\backslash \mathcal{T}$ is. Since $N$ acts freely on $\mathcal{T}$, the quotient graph of groups is a graph of trivial groups. Hence, the fundamental group $\pi_1(N\backslash \mathcal{T})$ of the graph $N\backslash \mathcal{T}$ can be identified with the group $N$. In particular, $N$ is free. Similarly, the image of $\pi_1((N\cap H)\backslash \mathcal{S})$ under the induced map can be identified with $N\cap H$. Hence, since the image of $(N\cap H)\backslash \mathcal{S} \to N\backslash \mathcal{T}$ is a connected subgraph, $N\cap H$ is a free factor of $N$.
\end{proof}

The situation in which we shall need to apply \cref{subtree} is as follows. 

\begin{cor}
\label{free_factor}
Let $G = F(S)*(*_{\alpha}G_{\alpha})$ with $F(S)$ the free group on the set $S$. Let $H = F(S')*(*_{\alpha}G'_{\alpha})\leqslant G$ with $S'\subset S$ and $G_{\alpha}'\leqslant G_{\alpha}$ for each $\alpha$. If $N\leqslant G$ is a subgroup that intersects each conjugate of each $G_{\alpha}$ trivially, then $N$ is a free group and $N\cap H$ is a free factor of $N$.
\end{cor}
\begin{proof}
The group $G$ is the fundamental group of a graph of groups with trivial edge groups, a loop edge for each $s\in S$ and with a vertex group $G_{\alpha}$ for each $\alpha$. Then $G$ acts on its Bass--Serre tree $\mathcal{T}$ so that each vertex stabiliser is either trivial or conjugate to some $G_{\alpha}$ and each edge stabiliser is trivial. By definition of $H$, the inclusion $H\leqslant G$ can be realised by a morphism of graphs of groups that is an inclusion at the level of graphs. Thus, taking the minimal $H$-invariant subtree $\mathcal{S}\subset \mathcal{T}$, we see that the induced map $H\backslash \mathcal{S}\to G\backslash \mathcal{T}$ is an inclusion at the level of graphs. Now \cref{subtree} applies directly.
\end{proof}
An immediate consequence of \cref{free_factor} is that the map $(N\cap H)_{\ab}\to N_{\ab}$ induced by inclusion is injective. We shall use this fact many times in the sequel.

\section{Group pairs}

\label{sect:gp}

 \subsection{Group pairs and the associated augmentation module}
By a {\bf group pair}, we understand a pair $\mathcal{P} = (G, X)$, where $G$ is a group and $X$ is a non-empty left $G$-set. In this paper, we will often assume that $X$ contains a marked element $x_0 = x_0(X)$ such that the $G$-orbit of $x_0$ is regular. This assumption is not needed for all of our results, but it simplifies the exposition considerably. We put  $X_0=X\setminus G\cdot x_0$.

The  {\bf augmentation ${R}G$-module $\omega_R(X)$} of a group pair $(G,X)$ is defined as the kernel of the canonical $RG$-homomorphism $R[X] \to R$. It is clear that we have that 
$$\omega_R(X)\cong R\otimes_{\Z} \omega_{\Z}(X).$$

Observe that if $X=G\cdot x_0$, then $\omega_R(X)$ is isomorphic to the augmentation ideal $I_{RG}$ of the group ring $RG$.

Let $\mathcal{Q} = (H, Y)$ and $\mathcal{P} = (G, X)$ be two group pairs. A {\bf map} $\kappa : \mathcal{Q} \to \mathcal{P}$ between group pairs is a pair of maps, consisting of a homomorphism $H \to G$ and a map $Y \to X$, both denoted by $\kappa$, sending the marked element of $Y$ to the marked element of $X$ and such that
\[
\kappa(h \cdot y) = \kappa(h) \cdot \kappa(y) \quad \text{for all } h \in H \text{ and } y \in Y.
\]
We denote by $\omega(\kappa)$ the induced map $\omega_R(Y)\to \omega_R(X)$.
If $\kappa$ is injective (in other words, if $H\to G$ and $Y\to X$ are injective), then we say that $\mathcal{Q}$ is a {\bf subpair} of $\mathcal{P}$. 
 
\subsection{An example: 2-complexes and  groups pairs}

A {\bf 2-complex} for us will be a 2-dimension CW-complex in which all attaching maps of 2-cells are immersions. Any 2-dimension CW-complex is homotopy equivalent to such a CW-complex. Following Wilton \cite{Wil24}, a {\bf branched morphism} of 2-complexes $Y\to X$ is a map which sends 0-cells to 0-cells, 1-cells homeomorphically to 1-cells and open 2-cells to open 2-cells via a branched cover with a single branch point in the centre. A branched morphism is a {\bf branched immersion} if it is an immersion (locally injective) away from the branch points in the centre of 2-cells.

There is a natural group pair associated with any finite connected 2-complex $X$, see \cite[Definition 3.2]{Wil24}. It is defined as follows. Let $F = \pi_1(X^{(1)})$ and let $w_1, \ldots, w_n\in F$ be the (conjugacy class representatives of) elements given by the attaching maps of 2-cells in $X$. Then the group pair $\mathcal{P}_X$ is
\[
\mathcal{P}_X = (F_X, \mathcal{A}_X)
\]
where $\mathcal{A}_X = \bigsqcup_{i=1}^nF/\langle w_i\rangle$.

Any branched morphism of finite connected 2-complexes $\phi\colon Y\to X$ gives rise to a natural map of group pairs $\phi_{\#}\colon \mathcal{P}_Y\to \mathcal{P}_X$. The following is a key observation of Wilton \cite[Lemma 3.4]{Wil24}.

\begin{lem}
\label{branched_immersion_subpair}
If $\phi\colon Y\to X$ is a branched immersion of 2-complexes, then the induced map of group pairs $\phi_{\#}\colon \mathcal{P}_Y\to \mathcal{P}_X$ is injective.
\end{lem}

The converse of \cref{branched_immersion_subpair} is not quite true, one has to consider the more general class of essential maps. See \cite{Wil24} for details.

\subsection{The relation module}
 
Let   $(G,X)$ be  a group pair  with a complete set of $G$-orbit representatives $T\subset X_0$ and a  subset $S\subset G$ such that
\begin{equation}\label{generation}
G = \left\langle S, \bigcup_{t\in T}G_t\right\rangle. 
\end{equation}

The {\bf relation module }  (relative to \cref{generation}) is the kernel of the map
\begin{equation}\label{mapalpha}
\alpha_{S,T}:(\oplus_{s\in S} RG\cdot e_s)\bigoplus( \oplus_{t\in T} {}^GI_{RG_t}\cdot e_t)\to I_{RG}
\end{equation}
that sends $e_s$ to $s-1$ and $e_t$ to 1. Notice that  $\ker \alpha_{S,T}$  is isomorphic to $R\otimes_\Z N_{\ab}$, where $N$ is the kernel of the canonical map 
$$F(S)*(*_{t\in T} G_t)\to G.$$

From the definition of $\omega_R(X)$, we obtain the exact sequence 
 
\begin{equation}\label{firstseq}
0\to   \oplus_{t\in T} {}^GI_{RG_t}\cdot e_t\xrightarrow{\gamma} I_{RG}\cdot e_0\bigoplus( \oplus_{t\in T} RG\cdot e_t)\xrightarrow{\tau_{S,T}} \omega_R(X)\to 0,\end{equation}
where $\gamma(e_t)=e_0+e_t$, $\tau_{S,T}(e_0)=x_0$ and $\tau_{S,T}(e_t)=t-x_0$.   Now consider the exact sequence
\begin{multline*}
0\to (\oplus_{s\in S} RG\cdot e_s)\bigoplus(\oplus_{t\in T} {}^GI_{RG_t}\cdot e_t)\xrightarrow{\gamma} \\ I_{RG}\cdot e_0\bigoplus( \oplus_{t\in T} RG\cdot e_t)\bigoplus(\oplus_{s\in S} RG\cdot e_s)\xrightarrow{\tau_{S,T}}
 \omega_R(X)\to 0,\end{multline*}
  where $\gamma (e_s)= (s-1)e_0-e_s$  and  $\tau_{S,T}(e_s)=(s-1)x_0$. Observe  that the composition of $\gamma$ and the projection on $I_{RG}\cdot e_0$ coincides with the map $\alpha_{S,T}$. Let $\gamma_{S,T}$ be the composition of $\gamma$ and the projection on $( \oplus_{t\in T} RG\cdot e_t)\bigoplus(\oplus_{s\in S} RG\cdot e_s)$. Therefore, we obtain the exact sequence 
\begin{equation}
    \label{relationmoduleseq}
0\to \ker \alpha_{S,T}  \xrightarrow{\gamma_{S,T}}(\oplus_{t\in T} RG\cdot e_t)\bigoplus (\oplus_{s\in S} RG\cdot e_s)\xrightarrow{\tau_{S,T}} \omega_R(X) \to 0. \end{equation}
\begin{pro} 
\label{relation_module1}
Let $\kappa \colon (G,Y) \to (G,X)$ be a map between group pairs that acts as the identity on $G$.  
Let $T_Y$ and $T_X$ be complete sets of $G$-orbit representatives of $Y_0$ and $X_0$, respectively, and assume that $\kappa(T_Y) \subset T_X$.  
Let $S \subset G$ be such that
\[
G = \left\langle S, \;\bigcup_{t \in T_Y} G_t \right\rangle. 
\]

Then there exists a commutative diagram with exact rows:
\[
\begin{tikzcd}[column sep=small, row sep=large]
0 \arrow[r] 
   & \ker \alpha_{S,T_Y} \arrow[r, "\gamma_{S,T_Y}"] \arrow[d] 
   & \bigl(\!\bigoplus_{t \in T_Y} RG \cdot e_t\bigr) \oplus \bigl(\!\bigoplus_{s \in S} RG \cdot e_s\bigr) 
        \arrow[r, "\tau_{S,T_Y}"] \arrow[d, "\delta"] 
   & \omega_R(Y) \arrow[r] \arrow[d, "\omega(\kappa)"] 
   & 0 \\
0 \arrow[r] 
   & \ker \alpha_{S,T_X} \arrow[r, "\gamma_{S,T_X}"] 
   & \bigl(\!\bigoplus_{t \in T_X} RG \cdot e_t\bigr) \oplus \bigl(\!\bigoplus_{s \in S} RG \cdot e_s\bigr) 
        \arrow[r, "\tau_{S,T_X}"] 
   & \omega_R(X) \arrow[r] 
   & 0
\end{tikzcd}
\]

where $\delta(e_s) = e_s$ for $s \in S$, and $\delta(e_t) = e_{\kappa(t)}$ for $t \in T_Y$.

 \end{pro}

 \begin{proof} The commutativity of the diagram follows from
the construction of the maps in the diagram.
 \end{proof}
\begin{rem}

The map $\ker \alpha_{S,T_Y}\to\ker \alpha_{S,T_X}$ obtained in Lemma \ref{relation_module1} can be also understood in the following way. Observe that $G_{t}\le G_{\kappa(t)}$ for every $t\in T_Y$. Thus, we have the commutative diagram
 
\[
\begin{tikzcd}
1 \arrow[r] & K \arrow[r] \arrow[d] & F(S)*(*_{t\in T_Y} G_t) \arrow[r] \arrow[d] & G \arrow[r] \arrow[d, "\Id"] & 1 \\
1 \arrow[r] & N \arrow[r] & F(S)*(*_{t\in T_X} G_t) \arrow[r] & G \arrow[r] & 1
\end{tikzcd}.
\]
The first vertical arrow induces the map $R\otimes_{\Z} K_{\ab}\to R\otimes_{\Z} N_{\ab}$. The natural identification of $R\otimes_{\Z} K_{\ab}$ with $\ker \alpha_{S,T_Y}$ and $R\otimes_{\Z} N_{\ab}$ with $\ker \alpha_{S,T_X}$ induces the map $\ker \alpha_{S,T_Y}\to\ker \alpha_{S,T_X}$ obtained in Lemma \ref{relation_module1}.

The importance of  Lemma \ref{relation_module1} is that we can see the map $R\otimes_{\Z} K_{\ab}\to R\otimes_{\Z} N_{\ab}$ as a restriction of a map between two free  $RG$-modules. 
\end{rem}
\subsection{Cohomological dimension of group pairs}
Following Alonso \cite{Al91}, the $R$-\textbf{cohomological dimension} of the pair $(G, X)$ is defined as  
\[
\cd_R(G, X) = \prd_{RG}(\omega_R(X)) + 1.
\]  

The following theorem arises from a result of Dicks \cite{Di80} when $X = G/H$ (the finitely generated case is due to Dunwoody \cite{Du79}). Alonso gave a different proof of it for $R=\mathbb{Z}$ in \cite[Theorem~3]{Al91}.

\begin{teo}[\cite{dicks_89}, Theorems~IV.4.8 and IV.4.11]\label{cd=1}
   Let $(G,X)$ be a group pair. Then 
   \[
   \cd_{R}(G, X) = 1 
   \]
   if and only if for all distinct $x, y\in X$, $|G_x\cap G_y|$ is invertible in $R$ and there exists a $G$-tree $\mathcal{T}$ with finite edge stabilisers having $X$ as a $G$-subset of $V(\mathcal{T})$ such that for every $v\in V(\mathcal{T}) \setminus X$, the highest common factor of 
   \[
   \bigl\{\,|G_v : G_u \cap G_v| : u \in X \,\bigr\}
   \]
   is invertible in $R$ (and, in particular, $G_v$ is finite).

   In particular, 
   \[
   \cd_{\mathbb{Z}}(G, X) = 1 
   \quad \Longleftrightarrow \quad
   G \cong F * \Bigl(*_{t\in T} G_t\Bigr),
   \]
   for some free group $F$ and some complete set of $G$-orbit representatives $T \subset X$. 
\end{teo}

The main example in this paper is the following. 
\begin{lem}\label{onerelatorpair} 
    Let $A$ and $B$ be two locally indicable groups, and let $u \in A * B$ be an element that is neither conjugated to an element in $A$ or $B$ nor a proper power. Let $n\in \N$ and $G = \frac{A * B}{\normal{w}}$ with $w = u^n$.  Define
    \[
    X = G \cdot x_0 \sqcup G / A \sqcup G / B \sqcup G / \langle u\rangle.
    \]
    Then $\omega_{\mathbb Z}(X)$ is a one-relator $\Z G$-module and $\operatorname{cd}_{\mathbb{Z}}(G, X) \leq 2$.
\end{lem}

\begin{proof}
From \cref{relationmoduleseq}, it follows that we need to establish that the kernel of the canonical map  
\[
\alpha: {}^G I_{\mathbb{Z} A} \oplus {}^G I_{\mathbb{Z} B} \oplus {}^G I_{\mathbb{Z} \langle u \rangle} \longrightarrow I_{\mathbb{Z} G}
\]
is free of rank $1$.

Consider the free product $\widetilde{G} = A * B$, and let $a \in {}^{\widetilde{G}} I_{\mathbb{Z} A}$ and $b \in {}^{\widetilde{G}} I_{\mathbb{Z} B}$ be the unique elements such that
\[
u - 1 = a + b.
\]
Let $\overline{a}$ and $\overline{b}$ denote the images of $a$ and $b$ in $\mathbb{Z}G$, respectively. Then the element
\[
\gamma = (\overline{a}, \overline{b}, 1 - u)
\]
belongs to $\ker \alpha$.

On the other hand, if $(a', b', c(u - 1)) \in \ker \alpha$, then 
\[
(a' + c\overline{a},\, b' + c\overline{b},\, 0) \in \ker \alpha,
\]
and so, by \cref{identity_theorem}, there is some $\beta \in \Z G$ such that $\beta \cdot u = \beta$ and 
\[
(a' + c\overline{a},\, b' + c\overline{b}) = \beta \cdot (\overline{a}, \overline{b}).
\]
Thus, $(a', b', c(u - 1)) \in \mathbb{Z}G \cdot \gamma$. Hence, $\ker \alpha = \mathbb{Z}G \cdot \gamma$.

Now assume $r \cdot \gamma = 0$ for some $r \in \mathbb{Z}G$. Since $r(\overline{a}, \overline{b}) = 0$, \cref{identity_theorem} implies that $r \in \mathbb{Z}G \cdot (u - 1)$, and since $r(u - 1) = 0$, we deduce that
\[
r \in \mathbb{Z}G \cdot (1 + u + \cdots + u^{n - 1}).
\]
Therefore, $r = 0$, and so $\ker \alpha$ is free of rank $1$.
\end{proof}

 \subsection{Finiteness properties for group pairs}

A group pair $(G, X)$ is {\bf finitely generated} if $G\backslash X$ is finite and there is a complete set of $G$-orbit representatives $T\subset X_0$ and finite subset $S\subset G$ such that
\[
G = \left\langle S, \bigcup_{t\in T}G_t\right\rangle .
\]

We say that the pair $(G, X)$ is {\bf finitely presented} if it is finitely generated as above and there is a finite subset $U\subset F(S)*(*_{t\in T}G_t)$ such that
\[
G \cong F(S)*(*_{t\in T}G_t)/\normal{U}.
\]
Note that if the group pair $(G, X)$ is finitely generated (respectively, finitely presented) and $G_x $ is finitely generated (respectively, finitely presented) for each $x\in X$, then $G$ itself is finitely generated (respectively, finitely presented).

For $n \geq 1 $ we say that a group pair $(G, X)$ is of {\bf type $\FP_n(R)$} if the $R G$-module $\omega_R(X)$ has type $\FP_{n-1}(R)$. 

We shall first need to convert some standard facts about groups to facts about group pairs.

\begin{lem}
\label{fg} Let $R$ be a ring.
The following are equivalent:
\begin{enumerate}
\item $(G, X)$ is finitely generated.
\item $(G, X)$ has type $\FP_1(R)$.
\end{enumerate}
\end{lem}

\begin{proof}
Note that if $X$ consists of a single regular orbit, then $\omega_R(X) \cong I_{RG}$, where $I_{RG}\leqslant R G$ is the augmentation ideal. In this case, the argument is standard and the result can be found in \cite{Br82}. The argument for the general case is almost identical.

If $(G, X)$ is finitely generated, let $T\subset X_0$ be a complete set of $G$-orbit representatives and let $S\subset G$ be a finite set such that $G = \left\langle S, \bigcup_{t\in T}G_t\right\rangle$. Then $\omega_R(X)\leqslant R[X]$ is finitely generated as an $R G$-module by the elements $$\{t - x_0, (s-1)\cdot x_0 \mid t \in T , s\in S \}.$$  

Conversely, if $\omega_R(X)$ is finitely generated, then it is generated by a finite subset 
$$\Sigma\subset\{t - x_0, (g-1)\cdot x_0 \mid t \in T , g\in G\}$$ with $T\subset X_0$ a complete set of $G$-orbit representatives. This immediately implies that $T$, and hence $G\backslash X$, is finite. Since $\Sigma$ is finite, there is a finite subset $S\subset G$ such that $\Sigma\subset  \{t - x_0, (s-1)\cdot x_0 \mid t \in T , s\in S\}$. Thus $G = \left\langle S, \bigcup_{t\in T}G_t\right\rangle$, and so $(G, X)$ is finitely generated.
\end{proof}

From \cref{relationmoduleseq} we obtain:
 
\begin{cor}
\label{relation_module2}
Let $(G, X)$ be a finitely generated group pair, let $T\subset X$ be a complete set of $G$-orbit representatives and let $S\subset G$ be a finite subset so that $G = \langle S,  \bigcup_{t\in T}G_t\rangle$. The following are equivalent:
\begin{enumerate}
\item $(G, X)$ has type $\FP_2(R)$.
\item If $N = \ker(F(S)*(*_{t\in T}G_t)\to G)$, then $R\otimes_{\Z} N_{\ab}$ is finitely generated as a $R G$-module.
\end{enumerate}
\end{cor}

From \cref{firstseq} we get two convenient results.

\begin{lem}
\label{fg_stabs}
Let $(G, X)$ be a group pair of type $\FP_2(R)$. The following are equivalent:
\begin{enumerate}
\item $G$ is finitely generated.
\item $G_x $ is finitely generated for each $x\in X$.
\end{enumerate}
\end{lem}

\begin{proof} 
If $G_x $ is finitely generated for each $x\in X$, then $G$ is finitely generated by \cref{fg}. Now suppose that $G$ is finitely generated. 

Consider the exact sequence \cref{firstseq}. By assumption, $\omega_R(X)$ is finitely presented and $T$ is  finite. Therefore, if $G$ is finitely generated, then the kernel $\oplus_{t\in T} {}^GI_{RG_t}\cdot e_t$ is finitely generated. Thus the groups  $G_t$ are  finitely generated as well.
\end{proof}

The second result upgrades $\FP_2(R)$ for group pairs to $\FP_2(R)$ for the group itself provided that all $G_t$ are of type $\FP_2(R)$.

\begin{cor}
\label{fp2_pair_to_fp2}
Let $(G, X)$ be a group pair of type $\FP_2(R)$. If $G_x $ has type $\FP_2(R)$ for each $x\in X$, then $G$ has type $\FP_2(R)$.
\end{cor}

\subsection{A criterion for \texorpdfstring{$\FP_2(R)$}{FP2(R)}}
We prove the following criterion for $\FP_2(R)$ of group pairs, generalising a criterion from \cite{JL23}.

\begin{teo}
\label{rel_hom_coherence}Let 
 $\mathcal P=(G, X)$ be a group pair with $\cd_{R}(G, X)\leqslant 2$ and let $\mathcal Q=(G, Y)$ be a subpair of type $\FP_1(R)$. Let $RG\hookrightarrow \D$ be an embedding into an Artinian ring and suppose that
\[
\length_{\D}\Tor_1^{R G}( \D, \omega_R(X))<\infty.
\]
Then there exists a group pair $\mathcal Q'$  of type $\FP_2(R)$ such that the embedding $\mathcal Q\subset \mathcal P$ factors through $\mathcal Q\subset \mathcal Q'\to  \mathcal P$.
\end{teo}

We first need a useful lemma.

\begin{lem}
\label{fg_submodule}
Let $S$ be a ring, $P$   a projective $S$-module and $M\leqslant P$ an $S$-submodule. Suppose that $S$ embeds in an Artinian ring $\D$ and that
  \[
  \length_{\D}\im(\D\otimes_SM\to \D\otimes_SP)<\infty.
  \]
Then $M$ is contained in a finitely generated $S$-submodule of $P$.
\end{lem}

\begin{proof}
Let $m_1, \ldots, m_n\in M$ be a finite set of elements such that if $M' = \sum_{i=1}^n Sm_i$, we have 
\[
\im(\D\otimes_S M'\to \D\otimes_SP)=\im(\D\otimes_SM\to \D\otimes_SP).
\]
Since $P$ is projective, there is a free $S$-module $F$ such that $F = P\oplus P'$. Since $M'$ is finitely generated, we have $F = F_1\oplus F_2$ for some free modules $F_1, F_2$ with $F_1$ finitely generated and with $M'\leqslant F_1$. Now consider the map 
\[
\tau\colon M/M'\to F/F_1
\]
Since $F/F_1 \cong F_2$, we have that $\im(\tau)$ is a submodule of a free module. Since $\im(\D\otimes_{S}M/M'\to \D\otimes_S F_2) = 0$, we conclude that $\D\otimes_S\im(\tau) = 0$. Hence, from the commutative diagram:
\[
\begin{tikzcd}
\im(\tau) \arrow[r, hook] \arrow[d] 
  & \mathcal{D} \otimes_S F/F_1  \\
\mathcal{D} \otimes_S \im(\tau) \arrow[r, equal] 
  &    \mathcal{D} \otimes_S \im(\tau) \arrow[u]
\end{tikzcd}
\]
follows that $\im(\tau) = 0$. This implies that $M\leqslant F_1$, and so $M\leqslant \pi_P(F_1)$ where $\pi_P\colon F\to P$ is the projection map. Since $F_1$ is finitely generated, so is $\pi_P(F_1)$.
\end{proof}

\begin{proof}[Proof of \cref{rel_hom_coherence}]
Choose  complete set of $G$-orbit representatives $T_Y\subset Y_0$ and $T_X\subset  X_0$ such that $ T_Y\subset T_X$.
Let $S$ be a subset of $G$ such that  $$
G = \left\langle S, \bigcup_{t\in T_Y}G_t\right\rangle. 
$$
  Consider the canonical map  
$$\phi: F(S)*(*_{t\in T_X}G_t) \to G$$ and let $N=(F(S)*(*_{t\in T_Y}G_t))\cap \ker \phi$.
By \cref{relation_module1},
we have the following commutative diagram

\begin{center}
\adjustbox{max width=\textwidth}{
\begin{tikzcd}
0 \arrow[r] & R\otimes_{\Z}N_{\ab} \arrow[d, "\iota_2"] \arrow[r, "\gamma_1"] &  (\oplus_{ s\in S} RG\cdot e_s)\bigoplus (\oplus_{ t\in T_Y} {}RG \cdot e_t) \arrow[r, "\tau_1"] \arrow[d, "\iota_1"] &  \omega_R(Y) \arrow[r] \arrow[d, "\iota_0"] & 0 \\
0 \arrow[r] & R\otimes_{\Z} (\ker \phi)_{\ab} \arrow[r, "\gamma"]                          & (\oplus_{ s\in S} RG\cdot e_s) \bigoplus (\oplus_{ t\in T_X} {}RG\cdot e_t)  \arrow[r, "\tau"]                          & \omega_R(X) \arrow[r]                        & 0
\end{tikzcd}
}
\end{center}

Since $\cd_R(G, X) \leqslant 2$, we have that $\prd_{RG}(\omega_R(X)) \leq 1$ and so 
\[
\ker(\tau)=R\otimes_{\Z} (\ker \phi)_{\ab}
\]
is a projective $R G$-module. Now, applying $\D\otimes_{R G}-$ to the above, we obtain the commutative diagram
\[
\begin{tikzcd}
                                                   & \D\otimes_{\Z G}\ker(\tau_1) \arrow[d, "\Id\otimes\iota_2"] \arrow[r, "\Id\otimes \gamma_1"] & (\oplus_{ s\in S} \D\cdot e_s)\bigoplus (\oplus_{ t\in T_Y} {}\D \cdot e_t) \arrow[d, ""] \\
  {\Tor_1^{R G}( \D, \omega_R(X))} \arrow[r] & \D\otimes_{RG}\ker(\tau) \arrow[r, "\Id\otimes \gamma "]                                    &(\oplus_{ s\in S} \D\cdot e_s)\bigoplus (\oplus_{ t\in T_X} {}\D \cdot e_t)                        
\end{tikzcd}.
\]
Using the commutativity of the diagram and the fact that
\begin{align*}
\length_{\D} (\oplus_{ s\in S} \D\cdot e_s)\bigoplus (\oplus_{ t\in T_Y} {}\D \cdot e_t)&<\infty \textrm{\ and\ }
\length_{\D}\Tor_1^{R G}(\D,\omega_R(X))<\infty
\end{align*}
we see that 
\[
\length_{\D}\im(\Id\otimes\iota_2)<\infty.
\]
By \cref{fg_submodule}, this implies that $\im(\iota_2)$ lies in a finitely generated submodule $M\leqslant \ker(\tau)$.
 
Now choose any finite set of elements $U\subset \ker(\phi)$ whose images in $R\otimes_{\Z}\ker(\phi)_{\ab}$ generate an $RG$-module containing $M$.  Since $U$ is finite,
there  is a finite collection of    finitely generated subgroups $G_1\leqslant G_{t_1},\ldots,G_n\leqslant G_{t_n} $ for some $\{t_1, \ldots, t_n\}\subset T_X\setminus \kappa(T_Y)$ such that 
\[
U\subset F' = F(S)*(*_{t\in T_Y} G_t)*(*_{i=1}^nG_i).
\]
Let $S^\prime$ be a finite generating set of $F^\prime$ over $F(S)*(*_{t\in T_Y}G_t)$. For each $s\in S^\prime$, let $g_s\in F(S)*(*_{t\in T_Y}G_t)$ be any element so that $\phi(s) = \phi(g_s)$. Then we have 
\[
U\subset K = \ker(\phi\mid F') = \normal{N, \{g_ss^{-1} \mid s\in S'\}}.
\]
Since $\ker(\phi)\cap G_t = 1$ for each $t \in T_X$, by \cref{free_factor} we see that $K$ is a free factor of $\ker(\phi)$. In particular, we may make identifications:
\[
R\otimes_{\Z}N_{\ab}\leqslant M\leqslant R\otimes_{\Z} K_{\ab}\leqslant R\otimes_{\Z}\ker(\phi)_{\ab}.
\]
But then this implies that $R\otimes_{\Z} K_{\ab}$ is generated (as an $RG$-module) by the images of $U$ and the images of $\{g_ss^{-1} \mid s\in S' \}$, and hence it is a finitely generated $RG$-module. Finally consider $Y'=Y\cup (\cup^n_{i=1} G/G_i)$ and $\mathcal Q'=(G, Y')$. Then it is clear that  the embedding $\mathcal Q\subset \mathcal P$ factors through $\mathcal Q\subset \mathcal Q'\to  \mathcal P$ and since $R\otimes_{\Z} K_{\ab}$ is finitely generated, $\mathcal Q'$ is of type $\FP_2(R)$ by \cref{relation_module2}.
\end{proof}

We have the following consequence of \cref{rel_hom_coherence}.

\begin{cor}\label{fgstabilisers} Let $\mathcal{P} = (G, X)$ be a group pair with $\operatorname{cd}_R(G, X) \leqslant 2$ and $G$ finitely generated. Let $RG \hookrightarrow \mathcal{D}$ be an embedding in an Artinian ring, and suppose that
\[
\length_{\D}\Tor_1^{RG}(\mathcal{D}, \omega_R(X))<\infty.
\]
Then:
\begin{enumerate}
    \item $G_x $ is finitely generated for every $x \in X$.
    \item If $\mathcal{Q} = (G, Y)$ is a subpair of $\mathcal{P}$ of type $\FP_1(R)$, then there exists a subpair $\mathcal{Q}' = (G, Y')$ of type $\FP_2(R)$ such that $\mathcal{Q}\subset\mathcal{Q}'\subset\mathcal{P}$.
\end{enumerate}
\end{cor}

\begin{proof}
    Let $x \in X_0$ and consider the $G$-set $Y = G \cdot x_0 \cup G \cdot x$. By \cref{rel_hom_coherence}, there exists a group pair $(G, Y')$ of type $\operatorname{FP}_2(R)$ such that $(G, Y) \hookrightarrow (G, Y')$. Then, by \cref{fg_stabs}, $G_x$ is finitely generated. This establishes the first statement.

    For the second statement, we note that the first statement implies that in the proof of \cref{rel_hom_coherence} we can take each $G_i$ to be equal to $G_{t_i}$ (as they are finitely generated) and so $\mathcal{Q}'$ can be taken to be a subpair of $\mathcal{P}$.
\end{proof}

 \section{Vanishing of second \texorpdfstring{$L^2$}--Betti numbers of group pairs and homological coherence}
 \label{sect:L2}
 
\subsection{\texorpdfstring{$L^2$}--Betti numbers of group pairs}

Given a group pair $(G,X)$ and $k\ge 1$ we define
$$b_k^{(2)}(G, X)=\beta^{\Q G}_{k-1} (\omega_{\Q}(X)).$$
The following proposition is an analog of \cite[Proposition 3.9]{JL23} for group pairs.

\begin{pro}\label{L2subgroup}
Let $(G,X)$ be a group pair and $n\ge 1$. Assume that $\cd_{\Q}(G,X)\le n$ and $b_n^{(2)}(G, X)=0$. Then for every subgroup $H$ of $G$, $b_n^{(2)}(H, X)=0$.
\end{pro}
\begin{proof}
  Since $\operatorname{cd}_{\mathbb{Q}}(G, X) \leq n$, the module $\Tor_{n-1}^{\mathbb{Q}[G]}(\mathcal{U}(G), \omega_{\mathbb{Q}}(X))$ is a projective $\mathcal{U}(G)$-module. Thus, 
\[
  \Tor_{n-1}^{\mathbb{Q}[G]}(\mathcal{U}(G), \omega_{\mathbb{Q}}(X)) = \{0\}
\quad \text{if and only if} \quad
b_n^{(2)}(G, X) = 0.
\]
The multiplicative map
$
\mathcal{U}(H) \otimes_{\mathbb{Q}[H]} \mathbb{Q}[G] \longrightarrow \mathcal{U}(G)$
is injective. Therefore, taking into account that $\cd_{\Q}(G,X)\le n$ and using the Shapiro Lemma, we obtain that 
\[
\Tor_{n-1}^{\mathbb{Q}[H]}(\mathcal{U}(H), \omega_{\mathbb{Q}}(X)) 
\cong 
\Tor_{n-1}^{\mathbb{Q}[G]}(\mathcal{U}(H) \otimes_{\mathbb{Q}[H]} \mathbb{Q}[G], \omega_{\mathbb{Q}}(X)) = \{0\}.
\]
This implies that $b_n^{(2)}(H, X) = 0$.
\end{proof}

\begin{pro}\label{onerelatorflat}
    Let $A$ and $B$ be two locally indicable groups, and let $w \in A * B$ be an element that is neither conjugated to an element of $A$ or $B$ nor a proper power. Set
    \[
    G = A * B / \normal{w}
    \quad \text{and} \quad
    X = G \cdot x_0 \sqcup G / A \sqcup G / B.
    \]
    Let $K$ be a field of characteristic zero. Then the $KG$-module $\mathcal{D}_{KG} \otimes_K \omega_K(X)^{op}$ is flat. In particular, 
    \[
    b_2^{(2)}(G, X) = 0.
    \]
\end{pro}

\begin{proof}
    Observe that $\omega_K(X)$ is a one-relator $KG$-module by \cref{onerelatorpair}. Thus, by \cref{preflat}, the module $\mathcal{D}_{KG} \otimes_K \omega_K(X)^{op}$ is flat. Thus, we have
    \[
    \Tor_1^{KG}(\D_{KG}, \omega_K(X)) \cong \Tor_1^{KG}(\D_{KG}\otimes_K\omega_K(X)^{op}, K) = \{0\}
    \]
    and so $b_2^{(2)}(G, X) = 0$ as claimed.
\end{proof}

The previous proposition leads to the following interesting property of torsion-free one-relator products of locally indicable groups.

\begin{teo}\label{teo:intersection}
    Let $A$ and $B$ be two locally indicable groups, and let $w \in A * B$ be an element that is neither conjugated to an element in $A$ or $B$ nor a proper power. 
    Then, for every finitely generated subgroup $H$ of $A * B / \normal{w}$, the intersections $H \cap A, H\cap B$ are finitely generated. 
\end{teo}

\begin{proof}
    Let $X = G \cdot x_0 \sqcup G / A \sqcup G / B$. By \cref{onerelatorpair}, $\cd_{\Q}(G,X)\le 2$. By \cref{onerelatorflat}, $b_2^{(2)}(G, X) = 0$. Therefore, by \cref{L2subgroup},  we also have $b_2^{(2)}(H, X) = 0$. Now,  we put $\D=\D_{\Q G}$ and apply \cref{fgstabilisers}.
\end{proof}

\subsection{Proof of Theorem \ref{cohomologicalcoherence}}
Let $H$ be a finitely generated subgroup. Since $\cd_\mathbb{Q}(G, X) \leq 2$,  $\cd_\mathbb{Q}(H, X) \leq 2$ as well. Consider the subpair $(H, H \cdot x_0)$ of $(H, X)$. We want to apply \cref{rel_hom_coherence} with $\mathcal{D} = \mathcal{R}_{\mathbb{Q} G}$.

By \cref{L2subgroup}, since $b_2^{(2)}(G, X) = 0$, we also have $b_2^{(2)}(H, X) = 0$. Therefore, $\Tor_1^{\mathbb{Q} H}(\mathcal{D}, \omega_\Q(X)) = 0$. 

By \cref{fgstabilisers}, there exists a subpair $(H, Y)\subset (H, X)$ of type $\FP_2(\Q)$ such that for every $y \in Y$, the stabiliser $H_y$ is a finitely generated subgroup of $G_y$.

Since $G_x$ is homologically coherent over $\Q$ for all $x\in X$, it follows that $H_y$ is  of type $\FP_2(\mathbb{Q})$ for all $y\in Y\subseteq X$. Therefore, by \cref{fp2_pair_to_fp2}, $H$ is of type $\FP_2(\mathbb{Q})$.

\section{Cohen-Lyndon property and  promotion of coherence from homological  coherence}

\label{sect:prom}
Given a group pair $\mathcal{P} = (G, X)$  we put
\[\St_{\mathcal P}=\{G_x  : x \in X\}, \,\,\, N_{\mathcal{P}} = \langle \St_{\mathcal{P}}\rangle\, \textrm{\ and\ }\,
\pi(\mathcal{P}) = G / N_{\mathcal{P}}.
\]
In the following, we understand $\St_{\mathcal{P}}$ not as a multiset but as a set, that is, if for $x, y \in X$, we have $G_x  = G_y$, then it appears only once in $\St_{\mathcal{P}}$.

Given a map $\kappa:\mathcal Q\to \mathcal P$ between group pairs, it induces the natural map $\pi_\kappa: \pi(\mathcal Q)\to \pi (\mathcal P)$.

\subsection{Cohen-Lyndon property} Let $\mathcal{P} = (G, X)$ be a group pair. We say that it satisfies the {\bf Cohen–Lyndon property} if there exists a complete set of $N_{\mathcal P}$-orbit representatives $T \subset \St_{\mathcal P}$ (where $G$ acts by conjugation on $\St_{\mathcal P}$) such that
 $N_{\mathcal P}=  \ast_{K\in T} K$. Note that for $K\in T$, we have $G_K = N_G(K)$, where $N_G(K)$ denotes the normaliser of $K$ in $G$.
 
Our definition is equivalent to \cite[Definition~3.13]{Su20}.  
In their terminology, the triple $\bigl(G, \{N_G(K)\}_{K\in T}, T\bigr)$ has the Cohen--Lyndon property.
 The following result, due to Edjvet--Howie~\cite[Theorem~1.1]{EH87}, provides a group pair that satisfies the Cohen--Lyndon property.
\begin{teo}
\label{one-relator_CL}
Let $A$ and $B$ be locally indicable groups and let $w\in A*B$ be a cyclically reduced word. Then $(A*B, A*B/\langle w\rangle)$ is Cohen--Lyndon. 
\end{teo}

Given a group pair $\mathcal{P}$ satisfying the Cohen--Lyndon property, the following proposition helps us control the kernel of the map $\pi(\mathcal{Q}) \to \pi(\mathcal{P})$ for a subpair $\mathcal{Q}$ of $\mathcal{P}$.
 
\begin{pro}
\label{CL_kernel}
Let $\mathcal{P} = (G, X)$ be a group pair satisfying the Cohen-Lyndon property and let $\mathcal{Q} = (H, Y)$ be a subpair with associated map $\kappa$. Then $\ker(\pi_{\kappa})$ splits as a free product of a free group and a collection of   $H_x $ for  some $x\in X\setminus Y$.
\end{pro}

\begin{proof} 
Since $\mathcal{P}$ is Cohen-Lyndon,   there exists a complete set of $N_{\mathcal P}$-orbit representatives $T \subset \St_{\mathcal P}$ such that
 $N_{\mathcal P}=  \ast_{K\in T} K$.

Let $L = N_{\mathcal P} \cap H$. Then, by the Kurosh subgroup theorem, $L$ is the free product of a free group and a collection $\mathcal{S}$ of subgroups $H_x $ for some $x \in X$. 

Then $\ker(\pi_{\kappa})$ is obtained from $L$ by quotienting out the free factors $H_x $ with $x \in Y$.
\end{proof}

We say that a group pair $(G, X)$ satisfies the {\bf finitely generated intersection property (f.g.i.p.)} if, for every finitely generated subgroup $H$ of $G$ and every $x \in X$, the stabiliser $H_x $ is finitely generated. We say it satisfies the {\bf strong f.g.i.p. (s.f.g.i.p.)} if it satisfies the f.g.i.p. and if for every finitely generated subgroup $H$ of $G$ there are finitely many $H$-orbits $H\cdot x\subset X$ such that $H_x\neq \{1\}$. Any group pair $(F,X)$ with a free group $F$   and $G_x $ finitely generated for each $x\in X$ satisfies the f.g.i.p. If additionally $X$ consists of finitely many (non-regular) $F$-orbits, it also satisfies the s.f.g.i.p.
The homological coherence of $\pi(\mathcal{P})$ can be promoted to   coherence using the Cohen-Lyndon property under the following circumstances. The next corollary is a generalization  of \cref{promotion}
\begin{cor}\label{teo: hom_coh_plus_CLA_coherence}
  Let $G$ be a coherent group, and let $\mathcal{P} = (G, X)$ be a group pair satisfying the f.g.i.p.\ and the Cohen-Lyndon property, and such that for every $x \in X$, the group $G_x $ is locally indicable. 
Then every subgroup of $\pi(\mathcal{P})$ of type $\mathrm{FP}_2(\mathbb{Q})$ is finitely presented.
\end{cor}
\begin{proof}
Let $\overline H$ be a  subgroup of $\pi(\mathcal{P})$ of type $\mathrm{FP}_2(\mathbb{Q})$. Then there exists a finitely generated subpair $\mathcal Q=(H,Y)\xrightarrow{\kappa} \mathcal P$ with $H$ finitely generated such that $\im \pi_{\kappa}=\overline H$ and $\Q\otimes_{\Z}(\ker \pi_{\kappa})_{\ab}=0$. By \cref{CL_kernel}, $\ker \pi_{\kappa}$ is the free product of a free group and some copies of $H_x$. 
Since $G$ satisfies the f.g.i.p., all the groups $H_x$ are finitely generated. 
Moreover, since each $G_x$ is locally indicable, $H_x$ has infinite abelianization if it is nontrivial.
Thus, $\ker \pi_{\kappa}=\{1\}$. Hence $\overline H\cong \pi(\mathcal Q)$ is finitely presented.
\end{proof}
\subsection{Proof of Theorem \ref{main}}
In the case where $w$ is a proper power, the theorem is proved in \cite{HH23}. Assume now that $w$ is not a proper power.

By \cref{one-relator_facts}(3), the group $G = A * B / \normal{w}$ is locally indicable. Therefore, by \cite{JL20}, it satisfies the strong Atiyah conjecture. Set
\[
X = G \cdot x_0 \sqcup G / A \sqcup G / B.
\]
By \cref{onerelatorpair}, we have $\operatorname{cd}_{\mathbb{Q}}(G, X) \leq 2$, and by \cref{onerelatorflat}, $b_2^{(2)}(G, X) = 0$. Applying \cref{cohomologicalcoherence}, we conclude that $G$ is homologically coherent over $\mathbb{Q}$.

Now consider the group pair $\mathcal{P} = (A * B, A * B / \langle w\rangle)$. Since $A$ and $B$ are coherent, their free product $A * B$ is coherent. By \cref{one-relator_CL}, the pair $\mathcal{P}$ is Cohen–Lyndon.   Therefore, applying \cref{promotion}, we conclude that $G \cong \pi(\mathcal{P})$ is coherent.

\subsection{The Cohen--Lyndon property and 2-complexes}

We say that a 2-complex $X$ has the Cohen--Lyndon property if its associated group pair $\mathcal{P}_X$ does.
We may use \cref{CL_kernel} to prove the following curious property of branched immersions of 2-complexes.

\begin{pro}
If $X$ is a 2-complex with the Cohen-Lyndon property and $\phi\colon Y\to X$ is a branched immersion, then $\ker(\phi_*)$ is free.
\end{pro}

\begin{proof}
Note that $\pi(\mathcal{P}_X) = \pi_1(X)$, $\pi(\mathcal{P}_Y) = \pi_1(Y)$ and $\pi_{\phi_{\#}} = \phi_*$. By \cref{branched_immersion_subpair} the induced map $\phi_{\#}$ is injective. Since $\pi_1(X)_x\cong \Z$ for all $x\in \mathcal{A}_X$, \cref{CL_kernel} implies that $\ker(\phi_*) = \ker(\pi_{\phi_{\#}})$ is a free group.
\end{proof}

\subsection {Cohen-Lyndon property over \texorpdfstring{$R$}{R}}

In this section we will introduce a variation of the Cohen-Lyndon property. 
To motivate the definition that we will introduce later, we present the following characterization of the Cohen Lyndon property.

\begin{pro}\label{CL_projective}
    Let $\mathcal P=(G,X)$ be a group pair. Then $\mathcal P$ has the Cohen-Lyndon property if and only if $\omega_{\Z}({\St_{\mathcal P}})$ is projective as a $\mathbb Z N_{\mathcal P}$-module.
\end{pro}
\begin{proof}
Put $N=N_{\mathcal P}$. Assume first that $(G,X)$ satisfies the Cohen--Lyndon property. Then  there exists a complete set of $N$-orbit representatives $T \subset \St_{\mathcal P}$ such that $N=  \ast_{K\in T} K$.
Observe that $St_N(K)=K$ for every $K\in T$. Therefore, 
by \cref{cd=1}, $\omega_{\mathbb{Z}}(\St_{\mathcal{P}})$ is a projective $\mathbb{Z}N$-module.

Conversely, if $\omega_{\mathbb{Z}}(\St_{\mathcal{P}})$ is a projective $\mathbb{Z}N$-module, then \cref{cd=1} implies that 
$N \cong F * \Bigl(*_{t\in T} N_t\Bigr)$
   for some free group $F$ and some complete set of $N$-orbit representatives $T \subset \St_{\mathcal P}$. Since the normal subgroup of $N$ generated by $\{t:t\in T\}$ coincides with $N$, we conclude that $F=1$ and $t=N_t$ for all $t\in T$. Thus,
$\mathcal{P}$ has the Cohen--Lyndon property.
\end{proof}

Let $\mathcal{P} = (G,X)$ be a group pair. We say that $\mathcal{P}$ satisfies the \textbf{Cohen--Lyndon property over $R$} if  
\begin{enumerate}
    \item $\omega_{R}(\St_{\mathcal{P}})$ is projective  as an $RN_{\mathcal{P}}$-module, and 
    \item $N_{N_{\mathcal{P}}}(K) = K$ for every $K \in \St_{\mathcal{P}}$. 
\end{enumerate}

\begin{rem}
    Note that, in the definition above, if $K\in \St_{\mathcal{P}}$ is infinite (for example, if $G$ is torsion-free), then $N_{N_{\mathcal{P}}}(K) = K$ for the following reason. \cref{cd=1} implies that $N_{\mathcal{P}}$ acts on a tree $\mathcal{T}$ with $X\subset V(\mathcal{T})$ and with finite edge stabilisers. Thus, if $g\in N_{N_{\mathcal{P}}}(K)$ and $x\in X$ so that $G_x = K$, then $K$ stabilises the path in $\mathcal{T}$ connecting the vertex $x$ with the vertex $g\cdot x$ in $\mathcal{T}$. Since edge stabilisers are finite, if $K$ is infinite we must have that $g\cdot x = x$ and so $g\in K$.
\end{rem}

\begin{rem}
Arguing as in the proof of \cref{CL_projective} and using \cref{cd=1}, one can show that if $G$ is torsion-free, then a group pair $\mathcal{P} = (G, X)$ has the Cohen–Lyndon property if and only if $\omega_{\mathbb{Q}}(\St_{\mathcal{P}})$ is projective as a $\mathbb{Q} N_{\mathcal{P}}$-module.
\end{rem}

\section{Graphs of group pairs and coherent-by-cyclic groups}
\label{sec:extensions}

In this section we will prove \cref{hyperbolic_extension}.
We first prove a general statement about graphs of groups pairs. We then use this to obtain a general criterion for (homological) coherence of extensions by $\Z$. In the last section we combine this criterion with several well-known facts about hyperbolic groups to prove the theorem. In this section   group pairs are   considered without marked points.

\subsection{A long exact sequence for graphs of group pairs}
\label{sec:graphs_of_group_pairs}

In this section we define graphs of group pairs and prove an analogue of \cref{Chiswell} in this setting.

If $(G, X)$ is a group pair, we say a subpair $(H, Y)\subset (G, X)$ is {\bf induced} if $g\cdot Y\cap Y = \emptyset$ for all $g\in G - H$. In other words, for each $y\in Y, g\in G$, we have that $g\cdot y\in Y$ if and only if $g\in H$. We say that a map of group pairs $\phi\colon(H, Y) \to (G, X)$ is induced if it is an inclusion and if $(\phi(H), \phi(Y))\subset (G, X)$ is induced.

We define a {\bf graph of group pairs} to be a tuple 
\[
\mathcal{G} = (\Gamma, \{(G_v, X_v)\}_{v\in V(\Gamma)}, \{(G_e, X_e)\}_{e\in E(\Gamma)}, \{\partial_e^{\pm}\})
\]
where here $\partial_e^{\pm}\colon (G_e, X_e)\to (G_{e^{\pm}}, X_{e^{\pm}})$ are induced maps of group pairs. The fundamental group of a graph of group pairs is the fundamental group of the underlying graph of groups structure.

We make the following observation which motivates the induced assumption on the maps of group pairs.

\begin{lem}
\label{sep_maps}
If $\phi\colon (H, Y)\to (G, X)$ is an induced map of group pairs and $R$ is a ring, the induced map
\[
RG \otimes_{R[H]}R[Y] = {}^GR[Y]\to R[X]
\]
given by $g\otimes y\mapsto g\cdot \phi(y)$ is injective.
\end{lem}

If $H\leqslant G$ are groups and $Y$ is an $H$-set, we may define a left $G$-set
\[
G\times_HY = G\times Y/\sim
\]
where $\sim$ is the equivalence relation given by $(g_1, y_1) \sim (g_2, y_2)$ if $g_2^{-1}g_1\in H$ and $g_2^{-1}g_1\cdot y_1 = y_2$. Note that we have
\[
{}^GR[Y] \cong R[G\times_HY]
\]
as left $RG $-modules.

\begin{teo}
\label{gogp_sequence}
Let $R$ be a ring and let $\mathcal{G} = (\Gamma, \{(G_v, X_v)\}_{v\in V}, \{(G_e, X_e)\}_{e\in E}, \{\partial_e^{\pm}\})$ be a graph of group pairs. Denoting by $G$ the fundamental group of $\mathcal{G}$, the following sequence is exact:
\[
\begin{tikzcd}
0 \arrow[r] & {\bigoplus_{e\in E^+}{}^G\omega_R(X_e)} \arrow[r, "\partial"] & {\bigoplus_{v\in V}{}^G\omega_R(X_v)} \arrow[r] & {\omega_R(X)} \arrow[r] & 0 
\end{tikzcd}
\]
where $\partial$ is the restriction of the map
\begin{align*}
\partial \colon\bigoplus_e{}^GR[X_e] &\to \bigoplus_v{}^GR[X_v],\\
s\otimes x_e&\mapsto  s\cdot t_e\otimes \partial_e^+(x_e)-s\otimes \partial_e^-(x_e)
\end{align*}
and where
\[
X = \left(\bigsqcup_{v}G\times_{G_v} X_{v}\right)/\sim
\]
is a left $G$-set. Here $\sim$ is the equivalence relation generated by 
\[
(g, \partial_e^-(x)) \sim (g t_e, \partial_e^+(x))
\]
for $e\in E^+$ and $(g, x)\in G\times_{G_e} X_e$.

\end{teo} 

\begin{proof}
Abusing notation, for each $e\in E^+$ denote by $\partial_e^{\pm}$ the map
\begin{align*}
{}^{G_{e^{\pm}}}R[X_e] &\to R[X_{e^{\pm}}]\\
g\otimes x &\mapsto g\cdot \partial_e^{\pm}(x)
\end{align*}
By \cref{sep_maps}, $\partial_e^{\pm}$ is injective. Since $RG$ is flat as an $RG_v$-module (for each $v\in V = V(\Gamma)$), the following map is also injective:
\[
\Id\otimes\partial_e^{\pm}\colon {}^GR[X_e]\to {}^GR[X_{e^{\pm}}].
\]
Similarly, by flatness of $RG$, we have exact sequences
\[
\begin{tikzcd}
0 \arrow[r] & {}^G\omega_R(X_e) \arrow[r] & {{}^GR[X_e]} \arrow[r, "\epsilon"] & {}^GR \arrow[r] & 0 \\
0 \arrow[r] & {}^G\omega_R(X_v) \arrow[r] & {{}^GR[X_v]} \arrow[r, "\epsilon"] & {}^GR \arrow[r] & 0
\end{tikzcd}
\]
Now consider the map
\begin{align*}
\partial \colon\bigoplus_{e\in E^+}{}^GR[X_e] &\to \bigoplus_{v\in V}{}^GR[X_v],\\
s\otimes x_e&\mapsto  s\cdot t_e\otimes \partial_e^+(x_e)-s\otimes \partial_e^-(x_e)
\end{align*}
as in the statement. By construction, we have that the following diagram commutes
\[
\begin{tikzcd}
{\bigoplus_{e\in E^+}{}^GR[X_e]} \arrow[r, "\partial"] \arrow[d, "\epsilon"] & {\bigoplus_{v\in V}{}^GR[X_v]} \arrow[d, "\epsilon"] \\
\bigoplus_{e\in E^+}{}^GR \arrow[r, "\delta"]                              & \bigoplus_{v\in V}{}^GR                             
\end{tikzcd}
\]
where $\delta$ is defined in \cref{Chiswell}.
We now show that $\partial$ is injective.

Let $0\neq r\in \bigoplus_{e\in E^+}{}^GR[X_e]$.  For each $e\in E^+$, let $T_e\subset G$ be a complete set of coset representatives for $G_e$. Using the fact that
\[
{}^GR[X_e] = \bigoplus_{t\in T_e}t\otimes R[X_e],
\]
we may write
\[
r = \sum_{e\in E^+}\sum_{t\in T_e}t\otimes r_{e, t}
\]
where $r_{e, t}\in R[X_e]$. Of course, all but finitely many $r_{e, t}$ are equal to 0 (and at least one is not equal to $0$). Recall that the Bass--Serre tree $\mathcal{T}$ for the graph of groups $\mathcal{G}$ is the tree with edge set $\bigsqcup_{e\in E^+}G/G_e$ and with vertex set $\bigsqcup_{v\in V}G/G_v$. Since $\mathcal{T}$ is a tree and the number of elements $r_{e, t}$ that are non-zero is finite, the collection of edges $\mathcal{E}=\{tG_e\mid r_{e, t}\neq 0\}\subset E(\mathcal{T})$ is contained in a finite subtree of $\mathcal{T}$. In particular, there is an edge $tG_e\in \mathcal{E}$ (a leaf in this subtree) such that either $(tG_e)^- = tG_{e^{-}}\neq st_fG_{f^+}, sG_{f^-}$ or $(tG_e)^+ = tt_eG_{e^+}\neq st_fG_{f^+}, sG_{f^-}$ for all other edges $sG_f\in \mathcal{E}$. Therefore, $\partial(r)$ either contains a non-zero $\partial_e^{-}(t)\otimes R[X_{e^{-}}]$ summand, namely $-\partial^-_e(t)\otimes\partial^-_e(r_{e, t})$ or it contains a non-zero $\partial_e^{+}(t)\cdot t_e\otimes R[X_{e^{+}}]$ summand, namely $\partial^+_e(t)t_e\otimes\partial^+_e(r_{e, t})$. Thus, $\partial(r)\neq 0$ and so $\partial$ is injective. 

We note that we have isomorphisms:
\begin{align*}
\bigoplus_{e\in E^+}{}^GR[X_e]  &\cong \bigoplus_{e\in E^+}R[G\times_{G_e} X_e]\\
\bigoplus_{v\in V}{}^GR[X_v]  &\cong \bigoplus_{v\in V}R[G\times_{G_v} X_v]
\end{align*}
as $RG$-modules. With this description, we see that 
\[
\bigoplus_{v\in V}R[G\times_{G_v} X_v]/\im(\partial)\cong R[X]
\]
where $X$ is as in the statement of the theorem. Combining all of the above, we obtain the following diagram with exact rows and columns:
\[
\begin{tikzcd}
            & 0 \arrow[d]                                                                & 0 \arrow[d]                                                    & 0 \arrow[d]                       &   \\
0 \arrow[r] & {\bigoplus_{e}{}^G\omega_R[X_e]} \arrow[r] \arrow[d] \arrow[r] & {\bigoplus_v{}^G\omega_R[X_v]} \arrow[d] \arrow[r] & {\omega_R[X]} \arrow[d] \arrow[r] & 0 \\
0 \arrow[r] & {\bigoplus_e{}^GR[X_e]} \arrow[r, "\partial"] \arrow[d]                    & {\bigoplus_v{}^GR[X_v]} \arrow[r] \arrow[d]        & {R[X]} \arrow[r] \arrow[d]        & 0 \\
0 \arrow[r] & \bigoplus_{e}{}^GR \arrow[r, "\delta"] \arrow[d]               & \bigoplus_v{}^GR \arrow[r, "\epsilon"] \arrow[d]          & R \arrow[r] \arrow[d]             & 0 \\
            & 0                                                                          & 0                                                              & 0                                 &  
\end{tikzcd}
\]
This completes the proof.
\end{proof}

\begin{cor}
\label{mayer_vietoris}
Let $R$ be a ring and let $\mathcal{G} = (\Gamma, \{(G_v, X_v)\}_{v\in V}, \{(G_e, X_e)\}_{e\in E}, \{\partial_e^{\pm}\})$ be a graph of group pairs. Denoting by $G = \pi_1(\mathcal{G}, T)$ and by $X$ the $G$-set from \cref{gogp_sequence}, for any left $RG$-module $M$ there are long exact sequences:
\begin{align*}
\ldots\to &\Tor_{n+1}^{RG}(\omega_R(X), M) \to \bigoplus_{e\in E^+}\Tor_{n}^{RG_e}(\omega_R(X_e), M) \to \bigoplus_{v\in V}\Tor_n^{RG_v}(\omega_R(X_v), M)\to \ldots\\
\ldots\rightarrow &\bigoplus_{e\in E^+}\Ext^{n}_{RG_e}(\omega_R(X_e), M) \rightarrow \Ext^{n+1}_{RG}(\omega_R(X), M) \rightarrow \bigoplus_{v\in V}\Ext^{n + 1}_{RG_v}(\omega_R(X_v), M)\rightarrow \ldots
\end{align*}
\end{cor}

As a corollary, we obtain the following bounds on the cohomological dimension.

\begin{cor}
\label{gp_dimension}
Let $(G, X)$ be the group pair as defined in \cref{gogp_sequence}. Then:
\begin{align*}
\cd_R(G, X) \leqslant \sup_{e, v}\{\cd_R(G_e, X_e) + 1, \cd_R(G_v, X_v)\}.
\end{align*}
\end{cor}

\subsection{Coherence of cyclic extensions}

We shall now apply the results of the last section to the special case of extensions with $\Z$.

First we prove a technical result which explicitly describes the group pair structure from \cref{gogp_sequence} for extensions of certain group pairs by $\Z$.

\begin{pro}
\label{suspensions_cd2}
Let $G\cong H\rtimes_{\psi}\Z$ and suppose that $H$ splits as a graph of groups $\mathcal{H} = (\Gamma, \{H_v\}, \{H_e\}, \{\partial_e^\pm\})$ with finite edge groups and with infinite vertex groups that do not split non-trivially over finite groups. Put  
\[
X=\{gH_v g^{-1}:v\in V(\Gamma), g\in G\}.
\]
Then $G$ acts on $X$ (via left conjugation) and $(G,X)$ splits as an HNN-extension (of group pairs) with vertex and edge group pair $(H, X)$. Moreover, we have 
\begin{align*}
\cd_{\Q}(G, X) &\leqslant 2,\\
b_2^{(2)}(G, X) &= 0,
\end{align*}
and $N_G(gH_vg^{-1})\cong H_v\rtimes \Z$ if a positive power of $\psi$ sends $H_v$ to a conjugate of itself, $N_G(gH_vg^{-1}) \cong H_v$ otherwise.
\end{pro}

\begin{proof}
Since each vertex group of $\mathcal{H}$ is a maximal infinite subgroup of $H$ that does not split non-trivially over a finite subgroup, we note that there is a bijection $\sigma\colon V\to V$ such that for each $v\in V = V(\Gamma)$ there is an element $h_v\in H$ such that $\psi(H_{v})^{h_{v}} = H_{\sigma(v)}$. In particular, this implies that $\psi$ is also a bijection of $X$ and so we have a well-defined map of pairs
\begin{align*}
\phi\colon(H, X) &\to (H, X)\\
			(h,x) &\mapsto (\psi(h),\psi(x)).
\end{align*}
This map is clearly separating and so we have a graph of group pairs
\[
\mathcal{G} = (\Lambda, \{(H, X)\}, \{(H, X)\}, \{\partial^{\pm}\})
\]
where $\Lambda$ has a single vertex and a single edge and where $\partial^- = \Id$ and $\partial^+ = \phi$. Now, the formula from \cref{gogp_sequence}, implies that this is a decomposition of $(G,X)$. 

Applying \cref{gp_dimension} to the pair $(G,  X)$ and using the fact that $\cd_{\Q}(H, X) \leqslant 1$ by \cref{cd=1}, we see that $\cd_{\Q}(G, X)\leqslant 2$ and $\beta_1^{\Q H}(\omega_{\Q}(X)) = 0$, so $\beta_1^{\Q G}(\omega_{\Q}(X)) = 0$ by \cref{vanishing}. Hence $b_2^{(2)}(G, X) = 0$.

The fact about stabilisers is clear from the definition of $X$.
\end{proof}

\begin{cor}
\label{extension_coherence}
Let $G\cong H\rtimes_{\psi}\Z$ be a group satisfying the weak Atiyah conjecture and suppose that $H$ has type $\FP_2(\Z)$. If all subgroups $N\rtimes\Z \cong K\leqslant G$ with $N$ finitely generated and one-ended are (homologically) coherent (over $\Q$), then $G$ is (homologically) coherent (over $\Q$).
\end{cor}

\begin{proof}
By Dunwoody's accessibility theorem \cite{dunwoody_85}, there is a finite graph of groups decomposition $\mathcal{H}$ for $H$ in which each edge group is finite and each vertex group is one-ended (and so do not split non-trivially with finite edge groups). By \cref{suspensions_cd2} we have that $\cd_{\Q}(G, X)\leqslant 2$ and $b_2^{(2)}(G, X) = 0$, where $X$ is the $G$-set from \cref{suspensions_cd2}. Since each $G_x\leqslant G$ is isomorphic to a semidirect product of the form $H_{v}\rtimes\Z$ for each $x\in X$ (the bijection $\sigma$ from \cref{suspensions_cd2} has finite order on each vertex), we see that $G_x$ is (homologically) coherent (over $\Q$) for each $x\in X$. Thus, we may apply \cref{cohomologicalcoherence} to conclude that $G$ is homologically coherent over $\Q$. If each $G_x$ is also coherent, then each vertex group $H_v$ is coherent. Since a graph of coherent groups with finite edge groups is coherent (by results of Karrass--Solitar \cite{KS70,KS71}), $H$ is coherent and so $G$ is coherent by \cite[Theorem 1.3]{JL23}.
\end{proof}

\subsection{Cyclic extensions of hyperbolic groups}

We here prove \cref{hyperbolic_extension}. \cref{extension_coherence} reduces the proof to the case of one-ended hyperbolic groups. This case can be handled using several well-known facts about hyperbolic groups and was first proven by Kropholler--Vidussi--Walsh \cite[Theorem 4.1]{KVW21} (in a much more general form). Since \cite{KVW21} was withdrawn, we include a proof for completeness.

\begin{pro}
\label{one_ended_extension}
Let $H$ be a (homologically) coherent (over $\Q$) one-ended hyperbolic group. Then any semidirect product $H\rtimes\Z$ is also (homologically) coherent (over $\Q$).
\end{pro}

\begin{proof}
Let $\psi\in \Aut(H)$ and let $G = H\rtimes_{\psi}\Z$ be the semidirect product. If $H$ is virtually $\Z$, then $G$ is virtually $\Z^2$ and so is coherent. If $H$ is a cocompact Fuchsian group, then $H$ has a finite index closed surface subgroup and so $G$ is virtually surface-by-$\Z$. By the Dehn--Nielsen--Baer Theorem, this implies that $G$ is virtually the fundamental group of the mapping torus of a surface homeomorphism. In particular, $G$ is virtually the fundamental group of a 3-manifold and so is coherent by Scott's theorem \cite{Sc73}.

Now assume that $H$ is not virtually $\Z$ or cocompact Fuchsian. From Bowditch's canonical JSJ-decomposition of $H$ \cite{Bo98} one can obtain a graph of groups decomposition $\mathcal{G}$ for $G$ as follows: since the automorphism $\psi$ induces an $H$-equivariant isomorphism of the JSJ tree $\mathcal{T}$ for $H$, the group $H\rtimes_{\psi}\Z$ thus also acts on $\mathcal{T}$ and so $\mathcal{G}$ is the quotient graph of groups for this action (we may need to subdivide an edge if the action inverts an edge). Now vertex and edge stabilisers of $\mathcal{T}$ as a $G$-tree are extensions by $\Z$ of vertex and edge stabilisers of $\mathcal{T}$ as an $H$-tree. In particular, $G$ admits a graph of groups decomposition with edge groups virtually $\Z^2$ (extensions of virtually $\Z$ groups by $\Z$) and with vertex groups extensions $H_v\rtimes\Z$ for $H_v$ a vertex group of the JSJ decomposition for $H$.

The JSJ-decomposition has three types of vertex groups: two ended groups (and so virtually $\Z$), maximal hanging Fuchsian groups (and so virtually free) or rigid groups relative to incident edge groups (and thus have finite outer automorphism group relative to the incident edge groups). If $H_v$ is a virtually free group, then $H_{v}\rtimes\Z$ has a finite index free-by-$\Z$ subgroup and so is coherent by Feighn--Handel \cite{FH99}. If $H_{v}$ is rigid relative to its incident edge groups, then any semidirect product $H_{v}\rtimes\Z\leqslant G$ has a finite index subgroup isomorphic to $H_{v}\times \Z$. In this case, if $H_v$ is (homologically) coherent (over $\Q$), then so is $H_v\times\Z$ since finitely generated subgroups are isomorphic to products of finitely generated subgroups of $H_v$ and subgroups of $\Z$. Thus, $H_v\rtimes\Z$ is (homologically) coherent (over $\Q$) if $H_v$ is.

We have shown that $G$ splits as a graph of groups with virtually abelian edge groups and (homologically) coherent (over $\Q$) vertex groups. Hence, $G$ is (homologically) coherent (over $\Q$) by results of Karrass--Solitar \cite{KS70,KS71}.
\end{proof}

\begin{rem}
    Since one-ended hyperbolic groups are co-Hopfian by a result of Moioli \cite[Theorem 1.0.7]{Mo13}, \cref{one_ended_extension} also holds for ascending HNN-extensions of $H$.
\end{rem}

\begin{proof}[Proof of \cref{hyperbolic_extension}]
\cref{extension_coherence} reduces the (homological) coherence (over $\Q$) of $G$ to coherence of semidirect products $N\rtimes\Z\leqslant G$ with $N\leqslant H$ finitely generated and one-ended. In fact, we only have to consider the one-ended subgroups $N$ that are vertex groups in Dunwoody's decomposition \cite{Du79}. Since all such groups are quasi-convex in $H$, they are also hyperbolic. Since $H$ is (homologically) coherent (over $\Q$) by assumption, \cref{one_ended_extension} implies that each $N\rtimes\Z$ is (homologically) coherent (over $\Q$). Hence, $G$ is (homologically) coherent (over $\Q$).
\end{proof}

\section{Coherence of group algebras} \label{sec: coh_group_alg}

\subsection{Modules for group pairs of cohomological dimension 2}

Let $\mathcal P=(G, X)$ be a group pair with $\cd_{K}(G, X)\leqslant 2$. In this section we derive structural properties of the modules over $KG$. 

Let $T\subset X_0$ be a complete set of $G$-orbit representatives  and   $S\subset G$ be such that
\[
	G = \left\langle S, \bigcup_{t\in T}G_t\right\rangle.
\]
The condition that $\cd_{K}(G, X)\leqslant 2$ has the following consequence.

\begin{lem}\label{extrestr}
	Let $M_1$ and $M_2$ be two  $KG$-modules. Then the natural maps 
	\[
		\Ext^k_{KG}(M_1, M_2)\to \prod_{t\in T} \Ext^k_{KG_t}(M_1, M_2)
	\]
	are isomorphisms for $k>2$ and surjective for $k=2$.  
\end{lem}
\begin{proof}
	According to \cite[Proposition III.2.2]{Br82} we have that the adjunction map
	\[
		\Ext^k_{KG}(M_1, M_2) \to \Ext^k_{KG}(K, \Hom_{KG}(M_1, M_2))
	\]
	is an isomorphism of $KG$-modules. Since $\cd_K(G, X) \leqslant 2$, $\omega_{K}(X)$ has projective dimension at most $1$. Thus, applying the $\Ext$ functor to the sequence $0 \to \omega_{K}(X) \to K[X] \to K \to 0$ we get that
	\[
		\Ext^k_{KG}(K, \Hom_{KG}(M_1, M_2)) \to \Ext^k_{KG}(K[X], \Hom_{KG}(M_1, M_2))
	\]
	is an isomorphism for $k > 2$ and surjective for $k = 2$. By Shapiro Lemma we conclude that
	\[
		\Ext^k_{KG}(K[X], \Hom_{KG}(M_1, M_2)) \cong \prod_{t\in T} \Ext^k_{KG_t}(K, \Hom_{KG}(M_1, M_2)).
	\]
	Hence the claim follows taking the inverse of the adjunction map for each $G_t$.
\end{proof}

The proof of the next result follows \cite[Corollary 2.3] {HH23}.

\begin{lem}\label{freeintersection}
	Let $g\in G$. If $t_1$ and $t_2$ are distinct elements of $T$, or if $g\not \in G_{t_2}$ , then   $\cd_K(G_{t_1}\cap G_{t_2}^g)\le 1$.
\end{lem}
\begin{proof}
	Put $H=G_{t_1}\cap G_{t_2}^g$ and let $L$ be a $KH$-module. We want to show that $\Ext^2_{KH}(K, L)=0$. By Shapiro Lemma,
	\[
		\Ext^2_{KH}(K, L)\cong \Ext^2_{KG}(K, \Hom_{KH}(KG,L)).
	\]
	Therefore, by \cref{extrestr}, we have an epimorphism
	\[
		\Ext^2_{KH}(K, L)\to \prod_{t\in T} \Ext^2_{KG_t}(K, \Hom_{KH}(KG,L))\cong 
		\prod_{t\in T} \prod_h \Ext^2_{K[H\cap (G_t)^u]}(K,  L),
	\]
	where $u$ ranges across double-coset representatives for $G_t \backslash G / H$. Note that $H=H\cap G_{t_1}=H\cap G_{t_2}^g$, and hence, the diagonal map
	\[
		\Ext^2_{KH}(K, L)\to \Ext^2_{KH}(K, L)\oplus \Ext^2_{KH}(K, L)
	\]
	must be surjective. Therefore, $\Ext^2_{KH}(K, L)=0$.
\end{proof}

The outcome of this section is that submodules of $KG$ admit an induced module structure from the stabilizers group algebras $KG_t$ up to a projective kernel.

\begin{pro} \label{pro: standard_modules}
	Let $I$ be a $KG$-submodule of a free module $KG^{\alpha}$. Then for each $t \in T$ there exists $I_t$ a $KG_t$-submodule of a free module such that $I$ admits a presentation of the form
	\[
		0 \to Q \to \oplus_{t \in T}\  ^G I_t \to I \to 0
	\]
	with $Q$ a projective $KG$-module. Moreover, for every $k \ge 1$  and every (respectively right) $KG$-module $L$ the natural maps
	\[
		\prod_{t \in T} \Ext^k_{KG_t}(I, L) \to \Ext^k_{KG}( \oplus_{t \in T}\  ^G I_t , L)
	\]
	and
	\[
		\Tor_k^{KG}(L,  \oplus_{t \in T}\  ^G I_t ) \to \bigoplus_{t \in T} \Tor_k^{KG_t}(L, I)
	\]
	are isomorphisms.
\end{pro}
\begin{proof}
	We put $\widetilde G=F(S)*(*_{t\in T} G_t)$ and $M=KG^{\alpha}/I$. Let $\widetilde I$ be the preimage of $I$ in $K\widetilde{G}^{\alpha}$. By \cite[Theorem 2.2]{Be74}, we have that there are $KG_t$-modules $I_t$   such that
	\[
		\widetilde I\cong \oplus_{t\in T}{}^{\widetilde G}I_t.
	\] 
	Moreover, by \cite[Proposition 2.1]{Be74}, $I_t$ is a  $KG_t$-submodules of $K\widetilde{G}^{\alpha}$ for each $t \in T$. We have the following exact sequence:
	\[
		0\to \widetilde I \to  K\widetilde{G}^{\alpha} \to M\to 0.
	\]
	Note that applying $KG\otimes_{K\widetilde  G}$ we obtain
	\[
		0\to \ker \tau \to KG\otimes_{K\widetilde G} \widetilde I \xrightarrow{\tau} I\to 0,
	\]
	and $KG\otimes_{K\widetilde G}\widetilde I\cong \oplus_{t\in T}{}^GI_t$. We shall show that $\ker \tau$ is a projective $KG$-module. More specifically, set $I' = \oplus_{t\in T}{}^GI_t$ and fix a $KG$-module $L$, we will prove that the natural map
	\[
		\Ext^k_{KG}(I, L) \to \Ext^k_{KG}(I', L)
	\]
	induced by $\tau$ is an isomorphism for $k \geq 2$ and surjective for $k = 1$.

	\begin{claim} \label{claim: ext_free_intersection}
		Let $t_1, t_2\in T$. Then for every $k\geq 1$,
		\[
			\Ext ^k_{KG_{t_2}} ({}^GI_{t_1},  L)\cong \left \{ \begin{array} {cc} 0 & t_1\ne t_2\\  \Ext^k_{KG_{t_2} }(I_{t_1},    L) & t_1=t_2\end{array} \right .
		\]
	\end{claim}
	\begin{proof}
		We have that, as a $KG_{t_2}$-module,
		\[
			{}^G I_{t_1} \cong \bigoplus_h KG_{t_2} \otimes_{K[G_{t_1} \cap G_{t_2}^h]} I_{t_1},
		\]
		where $h$ ranges over double coset representatives for $G_{t_2} \backslash G / G_{t_1}$. Observe that $I_{t_1}$ is a submodule of a free $K[G_{t_1}]$-module. Hence, by \cref{freeintersection}, for every $k \geq 1$ we have
		\[
			\Ext^k_{KG_{t_2}}\left(KG_{t_2} \otimes_{K[G_{t_1} \cap G_{t_2}^h]} I_{t_1}, L\right) = 0
		\]
		if $t_1 \ne t_2$ or $h \notin G_{t_2}$.
	\end{proof}
	
	The claim implies that if $t\in T$, then for $k \geq 1$ the canonical map
	\begin{equation} \label{eq: prime_to_t}
		\Ext^k_{KG_t}(I', L) \to \Ext^k_{KG_t}(I_t, L) ,
	\end{equation}
	is an isomorphism, and so
	\begin{equation} \label{passtoprime}
		\prod_{t\in T}\Ext^k_{KG_t}(I', L) \to \prod_{t\in T}\Ext^k_{KG_t}(I_t, L) \cong  \Ext^k_{KG}(I', L)
	\end{equation}
	is an isomorphism for $k \geq 1$.
	\begin{claim} Let $t\in T$. 
		The natural maps 
		\[
			\Ext^k_{KG_t}(I, L)\to   \Ext^k_{KG_t}(I', L)
		\]
		are isomorphisms for $k\geq 1$. 
	\end{claim}
	\begin{proof} 
		Consider the map $K\widetilde{G} \to KG$ as a morphism of $KG_t$-modules. Since $G_t$ is a subgroup of $G$, we can lift a right transversal for $G_t$ in $G$ to a right transversal for $G_t$ in $\widetilde{G}$, and hence, this map splits with a free kernel. Thus, the canonical map $\widetilde{I} \to I$ (viewed as a morphism of $KG_t$-modules) also splits with a free kernel. In particular, for $k \geq 1$ we obtain a canonical isomorphism
		\[
			\Ext^k_{KG_t}(I, L) \to \Ext^k_{KG_t}(\widetilde{I}, L).
		\]
		On the other hand, arguing with $\widetilde{G}$ instead of $G$ in \cref{claim: ext_free_intersection}, we get that for $k \geq 1$
		\[
			\Ext^k_{KG_t}(\widetilde{I}, L) \to \Ext^k_{KG_t}(I_t, L)
		\]
		is an isomorphism. Observe that as $KG_t$-modules we have the following commutative diagrams
		\[
		\begin{tikzcd}
			I' \arrow[rr, "{\tau}"]  &  & I \\
			& \widetilde{I} \arrow[lu, "{\iota}"] \arrow[ru] & 
		\end{tikzcd}
		\quad \mbox{and} \quad
		\begin{tikzcd}
			\widetilde{I} \arrow[rr, "{\iota}"]  &  & I' \\
			& I_t \arrow[lu] \arrow[ru] & 
		\end{tikzcd}
		\]
		where $\iota(x) := 1 \otimes x$ for $x \in \widetilde{I}$. Thus, composing the obtained isomorphisms with the inverse of \eqref{eq: prime_to_t} we get the isomorphism from the statement.
	\end{proof}
	
	The last claim, together with the isomorphism \eqref{passtoprime}, implies that for $k \geqslant 1$
	\[
		\prod_{t \in T} \Ext^k_{KG_t}(I, L) \to \Ext^k_{KG}(  I^\prime, L)
	\]
	is an isomorphism. A symmetric argument shows that for every (right) $KG$-module $L$ the natural maps
	\[
		\Tor_k^{KG}(L,  I') \to \bigoplus_{t \in T} \Tor_k^{KG_t}(L, I)
	\]
	are isomorphisms for $k \geq 1$. Finally, from \cref{extrestr} we conclude that the maps
	\[
	\Ext^k_{KG}(I, L) \to \Ext^k_{KG}(I', L)
	\]
	induced by $\tau$ are isomorphisms for $k \geq 2$ and surjective for $k = 1$.
\end{proof}

\subsection{Group algebras and group pairs}
In this section we prove the following theorem.
  \begin{teo}\label{coherencegroupalgebraspais}
  Let $\mathcal P=(G, X)$ be a group pair with $\cd_{K}(G, X)\leqslant 2$ and let   $KG\hookrightarrow \D$ be an embedding into an Artinian ring.  Suppose that $\D \otimes_K \omega_K(X)^{op}$ is flat as a $K[G]$ module. Assume that $K  G_x$ is coherent for all $x\in X$. Then $KG$ is coherent.  
  \end{teo}
  \begin{proof}
  	Let $I$ be a finitely generated $KG$-submodule of $KG$. The condition that $\D \otimes_K \omega_K(X)^{op}$ is flat has the following consequence.
  	
  	\begin{claim}\label{consflat}
	The canonical map
	\[
		\bigoplus_{t \in T} \Tor^{KG_t}_1(\D, I) \to \Tor^{KG}_1(\D, I)
	\]
	is an isomorphism.
	\end{claim}
	\begin{proof}
	Consider the short exact sequence of right $KG$-modules:
	\[
		0 \to \D \otimes_K \omega_K(X)^{op} \to \D \otimes_K K[X]^{op} \to \D \to 0.
	\]
	Applying $\otimes_{KG} (KG/I)$ and using that $\D \otimes_K \omega_K(X)^{op}$ is flat, we obtain that the map
	\begin{multline*}
	\Tor^{KG}_1\left(\D \otimes_K K[X]^{op}, I\right) \cong \Tor^{KG}_2\left(\D \otimes_K K[X]^{op}, KG/I\right) \to\\
	\Tor^{KG}_2(\D, KG/I) \cong \Tor^{KG}_1(\D, I)
	\end{multline*}
	is an isomorphism. By Shapiro Lemma, we have
	\[
		\Tor^{KG}_1\left(\D \otimes_K K[X]^{op}, I\right) 	\cong \bigoplus_{t \in T} \Tor^{KG_t}_1(\D, I).
	\]
	This completes the proof.
	\end{proof}
	
	According to \cref{pro: standard_modules} there are $KG_t$-modules $I_t$ and a projective $KG$-module $Q$ such that
	\[
		0 \to Q \xrightarrow{\gamma} I' \xrightarrow{\tau} I \to 0
	\]
	is exact, where $I' = \oplus_{t \in T} {}^G I_{G_t}$.
	
	\begin{claim}\label{trivialkernel}
		The map $\Id_{\D}\otimes \gamma: \D\otimes_{KG} Q \to \D\otimes_{K G}  I'$ is injective.
	\end{claim}
	\begin{proof}
	By \cref{pro: standard_modules} the canonical map
	\[
		\Tor_1^{KG}(\D,  {I'}) \to \bigoplus_{t \in T} \Tor_1^{KG_t}(\D, I)
	\]
	is an isomorphism. Combining this with \cref{consflat}, we obtain the desired result.
	\end{proof}
	
	Let $J$ be a finitely generated $KG$-submodule of $I'$ such that $\tau(J) = I$. Put $N = J \cap Q$. By Claim~\ref{trivialkernel}, we have the following commutative diagram with exact lower row:
	\[
	\begin{tikzcd}
	& \D \otimes_{KG} N \arrow[r] \arrow[d]  & \D \otimes_{KG} J \arrow[d] \\
	0 \arrow[r] & \D \otimes_{KG} Q \arrow[r, "{\Id_{\D} \otimes \gamma}"] & \D \otimes_{KG} I'
	\end{tikzcd}
	\]
	Therefore, the image of $\D \otimes_{KG} N$ in $\D \otimes_{KG}  Q$ has finite length. Thus, by \cref{fg_submodule}, we obtain that there exists a finitely generated submodule $N'$ such that $N \leq N' \leq Q$. In particular, there exist finitely many elements $t_1, \ldots, t_n \in T$ and finitely generated $KG_{t_i}$-submodules $I_i$ of $I_{t_i}$ such that $N', J \subseteq KG(I_1 + \ldots + I_n) = J'$. Note that
	\[
		J'\cong {}^GI_1\oplus\ldots \oplus {}^GI_n.
	\]
	Let $S$ be a finite generating set of the $KG$-module $J'$. For each $s \in S$, let $j_s \in J$ be such that $\tau(j_s) = \tau(s)$. Then $J' \cap Q$ is generated by $N$ and the set $\{s - j_s : s \in S\}$. Thus, $J' \cap Q$ is generated by $N'$ and $\{s - j_s : s \in S\}$, and hence is finitely generated. Now, since each $I_i$ is a finitely generated $KG_{t_i}$-submodule of a free module, by coherence, it is finitely presented. Therefore, putting all together we get that
	\[
		I \cong J' / (J' \cap Q)
	\]
	is finitely presented.
\end{proof}

\begin{proof}[Proof of \cref{teo:coherencegroupalgebras}] By \cref{one-relator_facts}(3), the group $G$ is locally indicable. 
    Let $X=G\cdot x_0\sqcup G/A\sqcup G/B$. By \cref{onerelatorpair}, $\cd_{K}(G,X)\le 2$ and $\omega_{K}(G,X)$ is one-relator module. By \cref{onerelatorflat}, the right $KG$-module $\mathcal D_{K G}\otimes_{K}\omega_{K}(X)^{op}$ is flat. Therefore by \cref{coherencegroupalgebraspais}, $K G$ is coherent.
\end{proof}

\subsection{Group algebras of hyperbolic-by-cyclic groups}
It is likely that \cref{hyperbolic_extension} also holds for group algebras. However, some ingredients are missing. For instance, we do not know the following, which, by contrast, is straightforward for groups.

\begin{Conj}
\label{conj:prod}
    If $H$ is a group so that $\Q[H]$ is coherent, then $\Q[H\times \Z]$ is coherent.
\end{Conj}

If \cref{conj:prod} were proven true, then by replacing the use of the results of Karrass--Solitar with results of Lam \cite{La77} and Aberg \cite{Ab82} and the result of Feighn--Handel with \cite[Theorem 3.4]{JL23} in the proof of \cref{one_ended_extension}, we may show that if $H$ is a one-ended hyperbolic group which is not cocompact Fuchsian, then $\Q[H\rtimes\Z]$ is coherent if $\Q[H]$ is. However, the case in which $H$ is a (virtual) surface group appears difficult and open.

\begin{Conj}
    If $\Sigma$ is a closed surface, and $G \cong \pi_1(\Sigma)\rtimes_{\psi}\Z$ for some $\psi\in \Out(\pi_1(\Sigma))$, then $\Q[G]$ is coherent.
\end{Conj}

\section{Farrell-Jones conjecture and group pairs} \label{sec: FJ_conj}

In this section, we prove the following theorem,  
which implies \cref{teoK_0_intro}.
\begin{teo} Let $G$ be a   group,    
$R$   a regular ring and $\mathcal{P}=(G,X)$ a group pair. Assume that
\begin{enumerate}
    \item $\cd_R(G)<\infty$ and the group ring $RG$ is coherent,
    \item    
$ (G, X)$ satisfies the Cohen-Lyndon property   over $R$,
    \item $\cd_R(N_G(G_x)/G_x)<\infty$ and the group ring $R[N_G(G_x)/G_x]$ is coherent for all $x\in X$. 
\end{enumerate}
    \label{teoK_0}
    Then the  natural map
    \[
        K_0(RG) \oplus \bigoplus_{x\in X} K_0\big(R[N_G(G_x)/G_x]\big) 
        \longrightarrow K_0\big(R[\pi(\mathcal{P})]\big)
    \]
    is surjective.
\end{teo}
In what follows, we assume that the hypotheses of the theorem hold.  
We divide the proof into several lemmas.

Let $\overline{M}$ be an $R[\pi(\mathcal{P})]$-module of type $\FP_1$.  
A \emph{lifting} of $\overline{M}$ is a finite resolution of an $RG$-module $M$ consisting of finitely generated projective $RG$-modules
\begin{equation}
    \label{resolutionM}
    0 \to L_n \xrightarrow{\tau_n} \dotsb \xrightarrow{\tau_2} L_1 \xrightarrow{\tau_1} L_0 \xrightarrow{\tau_0} M \to 0,
\end{equation}
such that
\[
    \overline{M} \cong R[\pi(\mathcal{P})] \otimes_{RG} M.
\]
Since $RG$ is coherent and, by \cref{critfinitepd}, any $RG$-module has finite projective dimension, a lifting of $\overline{M}$ always exists.  We refer to $n$ as the \emph{length} of the lifting in \cref{resolutionM}.
  We put $M_i = \operatorname{im} \tau_i$ and $\overline{M_i} = R[\pi(\mathcal{P})] \otimes_{RG} M_i$. Notice that  
\[
    0 \to L_n \xrightarrow{\tau_n} \dotsb \xrightarrow{\tau_{i+1}} L_i \xrightarrow{\tau_i} M_i \to 0
\]
is a lifting of $\overline{M_i}$.  
We call this lifting  an \emph{induced} lifting of $\overline {M_i}$.

 We have the following exact sequence:
\begin{equation}
\label{liftingM}    
0 \longrightarrow M_1 \longrightarrow L_0 \longrightarrow M \longrightarrow 0.
\end{equation}
Applying the functor $R[\pi(\mathcal{P})] \otimes_{RG} -$ to this sequence yields the long exact sequence
\[
\begin{aligned}
    0 \longrightarrow \Tor_1^{RG}\big(R[\pi(\mathcal{P})], M\big) 
    &\longrightarrow \overline{M_1}
    \longrightarrow R[\pi(\mathcal{P})] \otimes_{RG} L_0 \\
    &\longrightarrow \overline{M} \longrightarrow 0.
\end{aligned}
\]
Thus, it is natural to ask whether we can give a “nice” description
of the $R[\pi(\mathcal{P})]$-module $\Tor_1^{RG}\big(R[\pi(\mathcal{P})], M\big)$.  
We will do this using the fact that the group pair $\mathcal{P}$ satisfies the Cohen-Lyndon property  over~$R$.

Let $Y=\St_{\mathcal P}$. We  put   $N=N_{\mathcal P}$. By the definition of the Cohen-Lyndon property over $R$ we have that   for any subgroup $y\in Y$, $N_y=y$. Observe, that $Y$ is also a $G$-set, and so $R[Y]$ and $\omega_R(Y)$ are also $RG$-modules.  Observe that the $G$-stabiliser of a point in  $Y$ is the normilizer of the corresponding subgroup: $G_y=N_G(y)$.

\begin{lem}
\label{maindiagram}
There exists the following  commutative diagram with exact rows and columns:
\begin{equation}\label{diagramM}
\begin{tikzcd}[column sep=small, row sep=large]
0 \arrow[r] & \Tor^{RG}_1(R[\pi(\mathcal{P})], M) \arrow[r, "\beta _M"]  & \overline{M_1}    \\
0 \arrow[r] & \Tor^{RG}_1(R[\pi(\mathcal{P})], R[Y] \otimes_R M) \arrow[r] \arrow[u, "\alpha_{M}"'] & R[\pi(\mathcal{P})] \otimes_{RG} (R[Y] \otimes_R M_1) \arrow[u]\\
& 0 \arrow[u]
\end{tikzcd}
\end{equation}
   Moreover, if $M$ is a submodule of a free $RG$-module, then $\alpha_M$ is an isomorphism.
\end{lem}

\begin{proof}
The diagram is obtained by applying the functor $R[\pi(\mathcal{P})] \otimes_{RG} (-)$ to the following diagram of $RG$-modules:
\[
\begin{tikzcd}[column sep=large, row sep=large]
& 0   & & 0   &    \\
0 \arrow[r] & M_1 \arrow[r] \arrow[u]  & L_0 \arrow[r]   & M \arrow[r] \arrow[u] & 0 \\
0 \arrow[r] & R[Y] \otimes_R M_1 \arrow[r] \arrow[u] & R[Y] \otimes_R L_0 \arrow[r] \arrow[u] & R[Y] \otimes_R M \arrow[r] \arrow[u] & 0 \\
& \omega_R(Y) \otimes_R M_1 \arrow[u] & & \omega_R(Y) \otimes_R M   \arrow[u] &   \\
& 0 \arrow[u]& & 0 \arrow [u]&
\end{tikzcd}
\]
We   show that $\alpha_M$ is injective. There should appear $$\Tor_1^{RG}(R[\pi(\mathcal{P})],  \omega_R(Y)\otimes_R M)\cong \Tor_1^{RN}(R, \omega_R(Y)\otimes_R M)$$ and it is equal to zero because $\omega_R(Y)$ is flat as $RN$-module.

If $M$ is a submodule of a free $RG$-module $L_{-1}$, we obtain an exact sequence
\[
0 \longrightarrow \omega_R(Y) \otimes_R (L_{-1} / M)
\longrightarrow R[Y] \otimes_R (L_{-1} / M)
\longrightarrow L_{-1} / M \longrightarrow 0.
\]
As before, taking into account that $\omega_R(Y)$ is flat, we obtain that the natural map
\[
\alpha_{L_{-1}/M} \colon 
\Tor_2^{RG}\big(R[\pi(\mathcal{P})],\, R[Y] \otimes_R (L_{-1} / M)\big)
\longrightarrow
\Tor_2^{RG}\big(R[\pi(\mathcal{P})],\, L_{-1} / M\big)
\]
is an isomorphism.

Moreover, the dimension shifting provides the following commutative diagram, in which the vertical maps are isomorphisms:
\[
\begin{tikzcd}[column sep=large, row sep=large]
\alpha_{L_{-1}/M} \colon 
\Tor_2^{RG}\big(R[\pi(\mathcal{P})],\, R[Y] \otimes_R (L_{-1} / M)\big) 
\arrow[r] \arrow[d]
& \Tor_2^{RG}\big(R[\pi(\mathcal{P})],\, L_{-1} / M\big) \arrow[d] \\
\alpha_{M} \colon
\Tor_1^{RG}\big(R[\pi(\mathcal{P})],\, R[Y] \otimes_R M\big) 
\arrow[r]
& \Tor_1^{RG}\big(R[\pi(\mathcal{P})],\, M\big)
\end{tikzcd}
\]

This proves that $\alpha_M$ is an isomorphism.
\end{proof}
We say that the lifting \cref{resolutionM} of $\overline M$ is \emph{complete} if the map $\alpha_M$ from \cref{maindiagram} is an isomorphism. Observe that, by \cref{maindiagram}, any induced lifting is complete.

\begin{lem}\label{completeexists}
    Let $\overline M $ be an $R[\pi(\mathcal P)]$-module of type $\FP_2$. Then there exists a complete lifting for $\overline M $
\end{lem}
\begin{proof}
Assume that $\overline{M}$ is generated by $d$ elements. Put $L_0 = (RG)^d$. Since $\overline{M}$ is of type $\FP_1$, there exists a finitely generated $RG$-submodule $U_0 \subseteq L_0$ such that  
\[
\overline{M} \cong R[\pi(\mathcal{P})] \otimes_{RG} (L_0/U_0).
\]

Fix an $R$-basis $B$ of $L_0$. Enumerate the set
\[
S = \{(g-1)b : g \in \bigcup_{K \in Y} K,\, b \in B\} = \{u_i : i \in \mathbb{N}\}.
\] 
For $i \geq 1$ define $U_i = U_{i-1} + RG\,u_i$ and 
\[
U_i' = U_i \Big/ \sum_{j=1}^i RG\,u_j.
\]
Observe that
\[
L_0 \Big/ \sum_{j=1}^\infty RG\,u_j \;\cong\; \bigl(R[\pi(\mathcal{P})]\bigr)^d,
\qquad\text{and}\qquad
L_0 \Big/ \bigl(U_0 + \sum_{j=1}^\infty RG\,u_j \bigr) \;\cong\; \overline{M}.
\]

Let 
\[
U' = \frac{U_0 + \sum_{j=1}^\infty RG\,u_j}{\sum_{j=1}^\infty RG\,u_j}
\qquad\text{and}\qquad
\overline{U'} = R[\pi(\mathcal{P})] \otimes_{RG} U'.
\]
We then have 
\[
\overline{M} \cong \frac{(R[\pi(\mathcal{P})])^d}{\overline{U'}}.
\]
Thus, since $\overline{M}$ is of type $\FP_2$, the module $\overline{U'}$ is of type $\FP_1$.

Put
\[
\overline{U_i'} = R[\pi(\mathcal{P})] \otimes_{RG} U_i'.
\] 
Observe that
\[
\overline{U'} \cong \varinjlim \overline{U_i'}.
\]
We have that $\overline{U'}$ is of type $\FP_1$, the homomorphisms in this direct system are surjective and $  {U_0}$ is finitely generated. Therefore,  there exists $k$ such that $\overline{U'} \cong \overline{U_k'}$.  

Put $M_1 = U_k$ and $M = L_0/M_1$. We will now show that $\alpha_M$ is surjective.  

For this, we need to prove that the image of the composition $\beta_M \circ \alpha_M$ contains the image of $1 \otimes u_j$ ($j = 1, \ldots, k$) in 
\[
\overline{M_1} = \overline{U_k} = R[\pi(\mathcal{P})] \otimes_{RG} U_k.
\] 
Put $u = u_j$ and assume $u = (g-1)b$ with $g \in   G_x$ for some $x \in X$ and $b\in B$.  
Then we have the following commutative diagram:
\[
\begin{tikzcd}[column sep=small, row sep=large]
\Tor^{RG}_1(R[\pi(\mathcal{P})], M) \arrow[r, "\beta_M"] & \overline{M_1} &   \\
\Tor^{RG}_1(R[\pi(\mathcal{P})], R[Y] \otimes_R M) \arrow[r] \arrow[u, "\alpha_M"'] 
   & R[\pi(\mathcal{P})] \otimes_{RG} (R[Y] \otimes_R M_1) \arrow[u, "\gamma_M"] & \\
\Tor^{R[\langle g\rangle]}_1(R[\pi(\mathcal{P})], M) \arrow[r, "\delta_M"] \arrow[u, "\psi_M"] 
   & R[\pi(\mathcal{P})] \otimes_{R[\langle g\rangle]} M_1 \arrow[r, "\theta_M"] \arrow[u, "\sigma_M"] 
   & R[\pi(\mathcal{P})] \otimes_{R[\langle g\rangle]} L_0
\end{tikzcd}
\]

Here $\psi_M$ is the composition of the natural isomorphism
\[
\Tor_1^{R[\langle g\rangle]}(R[\pi(\mathcal{P})], M) 
   \longrightarrow \Tor_1^{RG}(R[\pi(\mathcal{P})], R[G/\langle g\rangle] \otimes_R M)
\]
that exists by Shapiro Lemma and the natural map
\[
\Tor_1^{RG}(R[\pi(\mathcal{P})], R[G/\langle g\rangle] \otimes_R M) 
   \longrightarrow \Tor^{RG}_1(R[\pi(\mathcal{P})], R[Y] \otimes_R M),
\]
coming from the compositions of the maps $R[G/\langle g\rangle] \to R[G/N_G(G_x)]\hookrightarrow R[Y]$.  

Now observe that 
\[
1 \otimes u \in \ker \theta_M = \operatorname{im} \delta_M.
\]
Thus, there exists $v \in \Tor_1^{R[\langle g\rangle]}(R[\pi(\mathcal{P})], M)$ such that $$\delta_M(v) = 1 \otimes u\in R[\pi(\mathcal{P})] \otimes_{R[\langle g\rangle]} M_1.$$ Then we obtain
\[
1 \otimes u 
= \gamma_M (\sigma_M(\delta_M(v)) )
= \beta_M (\alpha_M(\psi_M(v)))\in \overline{M_1}.
\]
This finishes the proof of the lemma.\end{proof}
The previous lemma suggests that we have to understand the structure of the $R[\pi(\mathcal{P})]$-module $\Tor^{RG}_1(R[\pi(\mathcal{P})], R[Y] \otimes_R M)$. This is done in the following lemma.
\begin{lem} \label{tor1}
Let $T\subset Y$ be a set of representatives of $G$-orbits.
Then we have that 
$$\Tor^{RG}_1(R[\pi(\mathcal{P})], R[Y] \otimes_R M)\cong \bigoplus_{t\in T} R[\pi(\mathcal{P})]\otimes_{R[G_t/N_t]} \Tor_1^{RN_t}(R, M).$$
\end{lem}
 \begin{proof} For each $x \in Y$, put $\overline{G_x} = G_x / N_x$.  
Then we have
\begin{align*}
   \Tor^{RG}_1\!\bigl(R[\pi(\mathcal{P})],\, R[Y] \otimes_R M \bigr) 
   &\;\cong\; \bigoplus_{t \in T} \Tor_1^{RG_t}\!\bigl(R[\pi(\mathcal{P})], M \bigr) \\[4pt]
   &\;\cong\; \bigoplus_{t \in T} R\!\left[ {\pi(\mathcal{P})}/{\overline{G_t}}\right] \otimes_R \Tor_1^{RN_t}(R, M) \\[4pt]
   &\;\cong\; \bigoplus_{t \in T} R[\pi(\mathcal{P})] \otimes_{R[\overline{G_t}]} \Tor_1^{RN_t}(R, M).
\end{align*}

 \end{proof}
\begin{lem}\label{fp-mod}
    Let $\overline{M}$ be an $R[\pi(\mathcal{P})]$-module of type $\FP$ and assume that \cref{resolutionM} is a complete lifting of $\overline M$.   Then we have that 
    \begin{enumerate}
        \item $\overline M_1$ is of type $\FP$ and
        \item for any $y\in Y$, the  $R[G_y/N_y]$-module $\Tor_1^{RN_y}(R, M)$ is of type $\FP$ and for all but finitely many   $G$-orbits in $Y$, $\Tor_1^{RN_y}(R, M)=0$.
    \end{enumerate}
\end{lem}
\begin{proof} Let $T\subset Y$ be a set of representatives of $G$-orbits.
    Combining \cref{maindiagram} and \cref{tor1}, we obtain the exact sequence
    $$0\to \bigoplus_{t\in T} R[\pi(\mathcal{P})]\otimes_{R[G_t/N_t]} \Tor_1^{RN_t}(R, M)\to \overline{M_1}\to R[\pi(\mathcal{P})]\otimes_{RG} L_0\xrightarrow{\gamma} \overline M\to 0.$$
Since $\overline M$ is of type $\FP$, $\ker \gamma$ is of type $\FP$. Thus, since $\overline M_1$ is of type $\FP_1$, by \cite[Proposition 1.4]{Bieri_Notes}, the $R[\pi(\mathcal{P})]$-module $$\bigoplus_{t\in T} R[\pi(\mathcal{P})]\otimes_{R[G_t/N_t]}  \Tor_1^{RN_t}(R, M)$$ is also of type $\FP_1$. Therefore, for each $t\in T$, the $R[G_t/N_t]$-module  $\Tor_1^{RN_t}(R, M)$ is of type $\FP_1$ and for all but finitely many   $t\in T$, $\Tor_1^{RN_t}(R, M)=0$. Since $\cd_R(G_t/N_t)<\infty$, $R$ is regular and the group ring $R[G_t/N_t]$ is coherent, \cref{critfinitepd} implies that $\Tor_1^{RN_t}(R, M)$ is of type $\FP$. Therefore, $\overline {M_1}$ is also of type $\FP$.
\end{proof}
Observe that, since any induced lifting is complete, using induction on $i$ and the previous lemma we obtain also that for every $1\le i\le n$,
  \begin{enumerate}
        \item $\overline M_i$ is of type $\FP$ and
        \item for any $y\in Y$, the  $R[G_y/N_y]$-module $\Tor_1^{RN_y}(R, M_i)$ is of type $\FP$ and for all but finitely many   $G$-orbits in $Y$, $\Tor_1^{RN_y}(R, M_i)=0$.
    \end{enumerate}

Now we are ready to prove \cref{teoK_0}.

\begin{proof}[Proof of \cref{teoK_0}]
 
Let 
\[
\theta \colon K_0(RG) \oplus \bigoplus_{x \in X} K_0\big(R[N_G(G_x)/G_x]\big) 
        \longrightarrow G_0\big(R[\pi(\mathcal{P})]\big)
\]
be the composition of the map
\[
K_0(RG) \oplus \bigoplus_{x \in X} K_0\big(R[N_G(G_x)/G_x]\big) 
        \longrightarrow K_0\big(R[\pi(\mathcal{P})]\big)
\]
and $\kappa_{R[\pi(\mathcal{P})]}$.  
Since, by \cref{K0G0}, $\kappa_{R[\pi(\mathcal{P})]}$ is an isomorphism, it suffices to show that $\theta$ is surjective.

Let $\overline{M}$ be an $R[\pi(\mathcal{P})]$-module of type $\FP$.  Our aim is to show that $[\overline M]\in \im \theta$.

By \cref{completeexists}, we may assume that \cref{resolutionM} is a complete lifting of $\overline{M}$.  
Put $M_0 = M$ and $\overline{M_0} = \overline{M}$.  
By \cref{fp-mod} and the remark after it,  all modules $\overline{M_i}$ ($0 \le i \le n$) are of type $\FP$.  
We will prove, by inverse induction on $i$, that $[\overline{M_i}] \in G_0(R[\pi(\mathcal{P})])$ belongs to the image of~$\theta$.

The base case $i = n$ is clear because $M_n \cong L_n$ is projective and finitely generated.  
Suppose we have proved that $[\overline{M_k}] \in \im \theta$ for all $k > i$.  
Then we have an exact sequence
\[
0 \longrightarrow M_{i+1} \longrightarrow L_{i+1} \longrightarrow M_i \longrightarrow 0,
\]
which, by \cref{maindiagram} and the remark after \cref{fp-mod}, induces an exact sequence
\begin{align*}
0 \;\longrightarrow\;& 
   \bigoplus_{j=1}^{k} R[\pi(\mathcal{P})] \otimes_{R[G_{t_j}/N_{t_j}]} 
   \Tor_1^{RN_{t_j}}(R, M_i) 
   \;\longrightarrow\; \overline{M_{i+1}} \\[4pt]
\longrightarrow\;& 
   R[\pi(\mathcal{P})] \otimes_{RG} L_i 
   \xrightarrow{\;\gamma\;} \overline{M_i} 
   \;\longrightarrow\; 0 .
\end{align*}

where, for each $j$, the $R[G_{t_j}/N_{t_j}]$-module
\[
U_j = \Tor_1^{RN_t}(R, M_i)
\]
is of type $\FP$.  
If $P_j$ is a finitely generated projective $R[G_{t_j}/N_{t_j}]$-module such that $[U_j] = [P_j]$ in $G_0(R[G_{t_j}/N_{t_j}])$, then
\[
\left [R[\pi(\mathcal{P})]\otimes_{R[G_{t_j}/N_{t_j}]} U_j\right ] 
= \left [R[\pi(\mathcal{P})]\otimes_{R[G_{t_j}/N_{t_j}]} P_j\right ]
\]
in $G_0(R[\pi(\mathcal{P})])$, because $G_{t_j}/N_{t_j}$ is a subgroup of $\pi(\mathcal{P})$.  
Hence $$\left [R[\pi(\mathcal{P})]\otimes_{R[G_{t_j}/N_{t_j}]} U_j\right ]$$ lies in the image of~$\theta$.  
By the inductive hypothesis, $[\overline{M_{i+1}}] \in \operatorname{im} \theta$.  We have also that $\left [R[\pi(\mathcal{P})]\otimes_{RG} L_i \right ]\in \im \theta$ because $L_i$ is projective and finitely generated.
Thus $[M_i] \in \im \theta$, and the proof is complete.
\end{proof}

\bibliographystyle{amsalpha}
\bibliography{bibliography}

\end{document}